\documentclass[11pt]{article}

\usepackage[utf8]{inputenc}

\usepackage[a4paper,margin=1.05in]{geometry}
\usepackage{amsmath,amssymb,amsthm,mathtools,mathrsfs}
\usepackage{microtype}
\usepackage{aliascnt}
\usepackage{hyperref}
\usepackage[nameinlink,noabbrev]{cleveref}
\usepackage[svgnames]{xcolor}
\usepackage{authblk}

\hypersetup{
  unicode=true,
  colorlinks=true,
  linkcolor=blue,
  citecolor=blue,
  urlcolor=blue,
  pdftitle={Stability and exponential convergence of Sinkhorn algorithm for entropy martingale optimal transport}
}

\theoremstyle{plain}

\newtheorem{theorem}{Theorem}[section]

\newaliascnt{lemma}{theorem}
\newtheorem{lemma}[lemma]{Lemma}
\aliascntresetthe{lemma}

\newaliascnt{proposition}{theorem}
\newtheorem{proposition}[proposition]{Proposition}
\aliascntresetthe{proposition}

\newaliascnt{corollary}{theorem}
\newtheorem{corollary}[corollary]{Corollary}
\aliascntresetthe{corollary}

\newaliascnt{assumption}{theorem}
\newtheorem{assumption}[assumption]{Assumption}
\aliascntresetthe{assumption}

\theoremstyle{definition}

\newaliascnt{definition}{theorem}
\newtheorem{definition}[definition]{Definition}
\aliascntresetthe{definition}

\theoremstyle{remark}

\newaliascnt{remark}{theorem}
\newtheorem{remark}[remark]{Remark}
\aliascntresetthe{remark}

\crefname{theorem}{Theorem}{Theorems}
\crefname{lemma}{Lemma}{Lemmas}
\crefname{proposition}{Proposition}{Propositions}
\crefname{corollary}{Corollary}{Corollaries}
\crefname{assumption}{Assumption}{Assumptions}
\crefname{definition}{Definition}{Definitions}
\crefname{remark}{Remark}{Remarks}
\crefname{equation}{Equation}{Equations}
\crefname{section}{Section}{Sections}

\title{Exponential convergence of Sinkhorn algorithm for entropy martingale optimal transport}
\author[1]{Anna Kazeykina}
\author[2]{Zhenjie Ren}
\author[1]{Hecheng Wang}
\affil[1]{LMO, Université Paris-Saclay}
\affil[2]{LaMME, Université \'Evry Paris-Saclay}
\date{}

\begin{document}

\maketitle

\begin{abstract}
We prove the exponential convergence in relative entropy of the Sinkhorn algorithm for the entropy martingale optimal transport.
We assume that the marginals have compact supports and are in strict convex order, the terminal
marginal support is convex, and the initial marginal support lies in the relative
interior of the terminal marginal support; the reference cost is assumed to be Lipschitz in each variable. Under these assumptions we
establish uniform bounds on the dual variables modulo affine gauges. This result allows us to obtain the existence and uniqueness of the optimizer, and its exponential representation in terms of dual variables. 
We then establish a relative-entropy
stability estimate for martingale couplings with different terminal marginals. Proving that the constant in the stability estimate stays uniform throughout the Sinkhorn iteration allows us to establish the exponential convergence of the Sinkhorn algorithm.
\end{abstract}

\tableofcontents

\section{Introduction}
The martingale optimal transport (MOT) problem was introduced in \cite{BeiglbockHenryLaborderePenkner2013ModelIndependentBounds, GalichonHenryLabordereTouzi2014StochasticControl}. It arises naturally in the calibration of stochastic volatility models under martingale constraints, which reflect no-arbitrage pricing conditions under the risk-neutral measure.

The stability problem is of particular interest for MOT because the marginal measures arising in financial applications are often noisy. This problem was studied in several works \cite{Wiesel2023ContinuityMOTRealLine, BackhoffVeraguasPammer2022StabilityMOTWeakOT}; in particular, \cite{Wiesel2023ContinuityMOTRealLine} proved value stability of the MOT problem with respect to the marginals, while the authors of \cite{BackhoffVeraguasPammer2022StabilityMOTWeakOT} proved stability of the optimal couplings with respect to the marginals and the cost in the weak topology.

Several numerical approaches were proposed for solving the martingale optimal transport problem. The paper \cite{GuoObloj2019ComputationalMOT} shows that the MOT problem can be approximated through a sequence of linear
programming (LP) problems which result from a discretization of the
marginal distributions combined with a relaxation of the martingale constraint; see also a continuation of this work in \cite{EcksteinGuoLimObloj2021RobustPricingMultipleAssets}, where deep neural-network optimization is incorporated into the dual problem. The authors of \cite{HiewNennaPass2025ODEEntropicOT} propose an ODE-based approach to solve multi-marginal entropic optimal transport problems with additional linear constraints, including martingale constraints.

The connection between entropy martingale optimal transport (EMOT) and
nonlinear pricing--hedging duality is developed in
\cite{DoldiFrittelli2023EMOT,DoldiFrittelliRosazzaGianin2024EntropyMOTTheory}.
These works incorporate martingale constraints into entropy optimal
transport and study the resulting nonlinear duality in mathematical
finance.
In \cite{CarlierMalamutSylvestre2025WeakOTMomentConstraints}, EMOT is put into a larger context of weak optimal transport problems with moment constraints, and the authors obtain general dual attainment results and study the convergence of the entropic regularization scheme; see also \cite{AlfonsiCoyaudEhrlacherLombardi2021MomentConstraintsOT} for approximation results and algorithms for moment-constrained OT.

The EMOT problem also appears under the name of the martingale Schrödinger bridge problem; see \cite{HenryLabordere2019MartingaleSchrodingerBridges, NutzWieselZhao2023MartingaleSchrodingerBridges, Guyon2024DispersionConstrainedMartingaleSchrodinger} for applications in finance and \cite{NutzWiesel2024MartingaleSchrodingerBridgeTwoDistributions} for a mathematical treatment.
In the one-dimensional setting with constant reference cost, \cite{NutzWiesel2024MartingaleSchrodingerBridgeTwoDistributions} proves existence and uniqueness of the optimizer and establishes its exponential representation in terms of three dual variables.
(Note that in the present paper we obtain an extension of these results to multiple dimensions and regular reference costs, but under a more restrictive assumption of compactness of marginal supports.)
Finally, see \cite{TanTouzi2013OptimalTransportationControlledDynamics, BenamouChazareixHoffmannLoeperVialard2024EntropicSemiMartingaleOT} for extensions to the semimartingale case and related numerical algorithms, including Sinkhorn algorithms.

Sinkhorn-type algorithms for EMOT are studied in
\cite{DeMarch2018EntropicApproximationMOT,DeMarchHenryLabordere2019ArbitrageFreeImpliedVolatility}.
The authors of~\cite{eckstein2025exponential} establish exponential
convergence for general iterative proportional fitting procedures using
strong convexity of the dual problem and geometric conditions on the
linear spaces defining the constraints. Note that the result of \cite{eckstein2025exponential} is proved under the assumption that the log-densities of the Sinkhorn iterates with respect to a given reference measure are uniformly bounded. One of the contributions of the present paper is the proof of this uniform boundedness property in the EMOT setting.

The recent paper \cite{TangShavlovskyRahmanianXiaoYing2025EntropicOTMartingaleType} considers discrete EMOT and studies Sinkhorn-type algorithms with sparse Newton iterations that utilize the approximate sparsity of the Hessian matrix of the dual objective. Numerical experiments indicate that the proposed algorithms converge rapidly and are robust with respect to total constraint violations.

The present work is primarily inspired by \cite{ChenConfortiRenWang2026SinkhornEMOT}. In that paper, the authors study the entropic regularization of the MOT problem and establish its dual formulation. They further prove exponential convergence of the associated Sinkhorn algorithm in terms of the dual variables. The goal of the present paper is to extend the approach of \cite{ChenConfortiRenWang2026SinkhornEMOT} beyond the one-dimensional case and to study convergence of the corresponding coupling iterates. We prove exponential convergence in relative entropy of the couplings of the martingale Sinkhorn iteration towards the optimal martingale coupling (Theorem \ref{thm:exponential-halfstep}). The strategy of the proof of exponential convergence via a stability estimate is inspired by \cite{chiarini2024semiconcavity}.

We consider the setting in which the marginals have compact supports and are in strict convex order, the terminal marginal support is convex,
and the initial support is contained in its relative interior. The reference cost is assumed to be Lipschitz in
each variable. We first prove
uniform bounds, up to affine gauge transformations, for the Sinkhorn
potentials. The key convexity argument uses the fact that each updated
terminal potential differs by a uniformly bounded function from a convex
log-partition function. 
A gauge-invariant decreasing functional then gives marginal convergence. Compactness of the gauged Sinkhorn potentials
produces the EMOT optimizer and its exponential representation in terms of Lipschitz dual potentials
(\cref{thm:optimizer_representation}).

We next prove a relative-entropy stability estimate for pairs of
martingale couplings (\cref{thm:stability}). Its hypotheses are a positive
lower bound on the density of one coupling relative to the product of
its marginals and boundedness of the logarithmic density of one coupling with respect to the other.
Cancellation of the initial-marginal and martingale terms reduces the
symmetrized entropy to the pairing of a terminal potential with the
terminal-marginal difference. An elementary estimate modulo affine
functions then bounds this entropy by the relative entropy of the
terminal marginals. The preceding uniform gauge bounds for dual potentials allow to prove that the conditions of the stability result are verified  along the Sinkhorn iteration. Combining stability with the
exact Sinkhorn half-step entropy decrement gives exponential convergence of the
martingale coupling iterates.

The paper is organised as follows. Section~2 introduces the notation,
assumptions, and states the main results. Section~3 collects the general entropy identities and an optimality result for
a martingale coupling admitting the exponential representation. Section~4 establishes well-posedness of the martingale Sinkhorn iteration, uniform
affine-gauge bounds for Sinkhorn potentials, marginal convergence, and existence of the optimizer. Section~5 proves the entropy stability estimate under
self-contained hypotheses. Section~6 verifies these hypotheses uniformly
along the iteration and derives the exponential convergence rate.

\section{Main results}\label{sec:main-result}

In this section we introduce the EMOT problem and the martingale
Sinkhorn iteration, then state the assumptions and main results:
uniform bounds on the terminal potentials modulo affine functions,
marginal convergence, existence of the optimizer, and exponential
convergence of the martingale half-steps
\cref{def:intermediate-measures}.

Let $\mathcal{P}(\mathbb{R}^d)$ denote the set of Borel probability measures on $\mathbb{R}^d$ for $d \geq 1$, and let
\[
\mathcal{P}_1(\mathbb{R}^d)
:=
\left\{
\mu\in\mathcal{P}(\mathbb{R}^d):
\int_{\mathbb{R}^d}|x|\,\mu(\mathrm{d}x)<\infty
\right\}.
\]

For $\alpha,\beta\in\mathcal{P}(\mathbb{R}^d)$, let $\Pi(\alpha,\beta)$ denote the set of couplings of $\alpha$ and $\beta$, that is,
\[
\Pi(\alpha,\beta)
:=
\left\{
\pi\in\mathcal{P}(\mathbb{R}^d\times\mathbb{R}^d):
\pi_{\# x}=\alpha,\;
\pi_{\# y}=\beta
\right\},
\]
where $\pi_{\# x}$ and $\pi_{\# y}$ denote the first and second marginals of $\pi$, respectively.
For $\alpha,\beta\in\mathcal{P}_1(\mathbb{R}^d)$, let $\mathcal{M}(\alpha,\beta)$ be the set of
martingale couplings in $\Pi(\alpha,\beta)$, namely
\[
\mathcal{M}(\alpha,\beta)
:=
\left\{
\begin{array}{l}
\pi\in\Pi(\alpha,\beta):\\[0.2ex]
\displaystyle \int \zeta(x)\cdot(y-x)\,\pi(\mathrm{d}x,\mathrm{d}y)=0\\[0.2ex]
\text{for every bounded measurable }\zeta:\mathbb{R}^d\to\mathbb{R}^d
\end{array}
\right\},
\]
where we write
$a\cdot b$ for the Euclidean inner product of $a,b\in\mathbb{R}^d$.

Let $\mu,\nu\in\mathcal{P}_1(\mathbb{R}^d)$. We write $\mu \preceq_{\mathrm{cx}} \nu$ if
\[
\int_{\mathbb{R}^d} f\,\mathrm{d}\mu
\le
\int_{\mathbb{R}^d} f\,\mathrm{d}\nu
\]
for every convex function $f:\mathbb{R}^d\to\mathbb{R}$ for which the two integrals are finite. By Strassen's theorem~\cite{Strassen1965}, for $\mu,\nu\in\mathcal{P}_1(\mathbb{R}^d)$, $\mu\preceq_{\mathrm{cx}}\nu$ is equivalent to $\mathcal{M}(\mu,\nu)\neq\varnothing$.
In particular, since affine functions are both convex and concave, $\mu$ and $\nu$ have the same barycenter:
\[
\int_{\mathbb{R}^d} x\,\mu(\mathrm{d}x)
=
\int_{\mathbb{R}^d} y\,\nu(\mathrm{d}y).
\]

For two probability measures $\alpha$ and
$\beta$, their relative entropy, or Kullback--Leibler divergence, is
\[
H(\alpha\mid\beta)
:=
\begin{cases}
\displaystyle \int \log\!\left(\frac{\mathrm{d}\alpha}{\mathrm{d}\beta}\right)\,\mathrm{d}\alpha,
& \text{if }\alpha\ll\beta,\\[1.2ex]
+\infty,&\text{otherwise.}
\end{cases}
\]

Let $R\in\mathcal{P}(\mathbb{R}^d\times\mathbb{R}^d)$ be a reference measure. 
The entropy martingale optimal transport problem is
\begin{equation}
\inf_{\pi\in\mathcal{M}(\mu,\nu)} H(\pi\mid R).
\label{eq:emot}
\end{equation}
A \emph{feasible coupling} for \eqref{eq:emot} is a measure
\(\pi\in\mathcal M(\mu,\nu)\), satisfying both prescribed marginals and
the martingale constraint. It is a \emph{finite-entropy feasible coupling}
if, in addition, \(H(\pi\mid R)<\infty\).
Whenever \eqref{eq:emot} admits a minimizer, we denote it by $\pi^\star$.

The following assumptions specify the marginals and reference measure.
The compactness and regularity conditions needed for the main results
will be stated separately.

\begin{assumption}[Marginals and reference measure]
\label{ass:main}
\leavevmode
\begin{enumerate}
\item The marginals $\mu,\nu\in\mathcal{P}_1(\mathbb{R}^d)$ satisfy $\mu\preceq_{\mathrm{cx}}\nu$.
\item The reference measure $R$ admits a strictly positive density with
respect to $\mu\otimes\nu$. Equivalently, there exists a measurable function
$c:\mathbb{R}^d\times\mathbb{R}^d\to\mathbb{R}$ such that
\begin{equation}
\frac{\mathrm{d}R}{\mathrm{d}(\mu\otimes\nu)}(x,y)
=
\exp\bigl(-c(x,y)\bigr)>0
\qquad
\text{for }(\mu\otimes\nu)\text{-almost every }(x,y).
\label{eq:reference-density}
\end{equation}
\end{enumerate}
\end{assumption}
We call \(c\) the reference cost. Since \(R\) is a probability measure,
\[
\int e^{-c(x,y)}\,\mu(\mathrm dx)\nu(\mathrm dy)=1.
\]
The reference measure need not itself have marginals \(\mu\) and \(\nu\).

We first write the constraint equations for a coupling admitting an exponential representation. This representation motivates the Sinkhorn iteration introduced further; the
existence of such a representation for the optimizer is proved in
\cref{thm:optimizer_representation}.

Under \cref{ass:main}, let $\pi\in\mathcal{M}(\mu,\nu)$, and suppose that there exist measurable functions
$\varphi,\psi:\mathbb{R}^d\to\mathbb{R}$ and $h:\mathbb{R}^d\to\mathbb{R}^d$ such that
\begin{equation}
\label{eq:exp_representation}
\frac{\mathrm{d}\pi}{\mathrm{d}R}(x,y)
=
\exp\bigl(-\varphi(x)-\psi(y)-h(x)\cdot(y-x)\bigr).
\end{equation}
Then
\begin{equation*}
\frac{\mathrm{d}\pi}{\mathrm{d}(\mu\otimes\nu)}(x,y)
=
\exp\bigl(-\varphi(x)-\psi(y)-h(x)\cdot(y-x)-c(x,y)\bigr),
\end{equation*}
and the martingale and marginal constraints yield
\begin{align*}
x
&=
\int_{\mathbb{R}^d}
y\,
\exp\bigl(-\varphi(x)-\psi(y)-h(x)\cdot(y-x)-c(x,y)\bigr)
\,\nu(\mathrm{d}y)
\quad\text{for }\mu\text{-almost every }x,
\\
1
&=
\int_{\mathbb{R}^d}
\exp\bigl(-\varphi(x)-\psi(y)-h(x)\cdot(y-x)-c(x,y)\bigr)
\,\mu(\mathrm{d}x)
\quad\text{for }\nu\text{-almost every }y,
\\
1
&=
\int_{\mathbb{R}^d}
\exp\bigl(-\varphi(x)-\psi(y)-h(x)\cdot(y-x)-c(x,y)\bigr)
\,\nu(\mathrm{d}y)
\quad\text{for }\mu\text{-almost every }x.
\end{align*}
Moreover, if $(X,Y)\sim\pi$, then a version of the conditional law of $Y$ given $X=x$ is
\begin{equation*}
\pi_x(\mathrm{d}y)
=
\exp\bigl(-\varphi(x)-\psi(y)-h(x)\cdot(y-x)-c(x,y)\bigr)\,\nu(\mathrm{d}y)
\quad\text{for }\mu\text{-almost every }x.
\end{equation*}
If a coupling admits the representation \eqref{eq:exp_representation},
we refer to \(\varphi,\psi,h\) as its dual variables or dual potentials.

We next introduce the martingale Sinkhorn algorithm for solving the EMOT
problem~\eqref{eq:emot}.
Steps~1 and~2 together impose the first marginal and the martingale
constraint. Step~3 imposes the second marginal; it need not preserve the
first marginal or the martingale constraint.

\begin{definition}[Martingale Sinkhorn iteration]
\label{def:martingale-sinkhorn}
Initialize the iteration by
\[
h^0=0,
\qquad
\varphi^0=0,
\qquad
\psi^0=0.
\]
For $n\geq 0$, define $(h^{n+1},\varphi^{n+1},\psi^{n+1})$ successively as
follows.

\smallskip
\noindent\textbf{Step 1.} For $\mu$-almost every $x$, choose
$h^{n+1}(x)$ as a solution of
\begin{equation}
0
=
\int_{\mathbb{R}^d}
(y-x)
\exp\bigl(-\varphi^n(x)-\psi^n(y)-h^{n+1}(x)\cdot(y-x)-c(x,y)\bigr)
\,\nu(\,\mathrm{d} y).
\label{eq:h-update}
\end{equation}

\smallskip
\noindent\textbf{Step 2.} For $\mu$-almost every $x$, set
\begin{equation}
\varphi^{n+1}(x)
=
\log\!\left(
\int_{\mathbb{R}^d}
\exp\bigl(-\psi^n(y)-h^{n+1}(x)\cdot(y-x)-c(x,y)\bigr)
\,\nu(\,\mathrm{d} y)
\right).
\label{eq:varphi-update}
\end{equation}

\smallskip
\noindent\textbf{Step 3.} For $\nu$-almost every $y$, set
\begin{equation}
\psi^{n+1}(y)
=
\log\!\left(
\int_{\mathbb{R}^d}
\exp\bigl(-\varphi^{n+1}(x)-h^{n+1}(x)\cdot(y-x)-c(x,y)\bigr)
\,\mu(\,\mathrm{d} x)
\right).
\label{eq:psi-update}
\end{equation}
\end{definition}

Proposition~\ref{prop:compact-well-posedness} proves that these updates
are well defined under extra assumptions formulated in \cref{ass:two-step}.
Definition~\ref{def:entropy-admissible-iteration} records the conditions on Sinkhorn iterates
needed for some general entropy identities.

% We keep separate notation for the two half-steps, since they satisfy
% different constraints.

\begin{definition}[Intermediate measures / half-steps]
\label{def:intermediate-measures}
For every $n\geq 0$, define the half-step measures
\begin{equation*}
\frac{\mathrm{d}\pi^{n,n}}{\mathrm{d}(\mu\otimes\nu)}(x,y)
=
\exp\bigl(-\varphi^n(x)-\psi^n(y)-h^n(x)\cdot(y-x)-c(x,y)\bigr),
\end{equation*}
and
\begin{equation*}
\frac{\mathrm{d}\pi^{n+1,n}}{\mathrm{d}(\mu\otimes\nu)}(x,y)
=
\exp\bigl(-\varphi^{n+1}(x)-\psi^n(y)-h^{n+1}(x)\cdot(y-x)-c(x,y)\bigr).
\end{equation*}
We call the update \(\pi^{n,n}\mapsto\pi^{n+1,n}\)
(Steps~1--2) the \emph{martingale half-step}, and
\(\pi^{n+1,n}\mapsto\pi^{n+1,n+1}\) (Step~3) the
\emph{terminal-marginal half-step}. We use the same names for the
resulting measures.
The one-sided constraint sets are defined by
\begin{equation*}
\begin{aligned}
\mathcal{M}(\mu,\cdot)
&:=
\left\{
\begin{array}{l}
\pi\in\mathcal{P}_1(\mathbb{R}^d\times\mathbb{R}^d):\ \pi_{\#x}=\mu,\\[0.2ex]
\displaystyle
\int \zeta(x)\cdot(y-x)\,\pi(\,\mathrm{d}x,\,\mathrm{d}y)=0\\[0.2ex]
\text{for every bounded measurable }\zeta:\mathbb{R}^d\to\mathbb{R}^d
\end{array}
\right\},
\\[1ex]
\Pi(\cdot,\nu)
&:=
\left\{
\pi\in\mathcal{P}(\mathbb{R}^d\times\mathbb{R}^d):\ \pi_{\#y}=\nu
\right\}.
\end{aligned}
\end{equation*}
The martingale constraint in \eqref{eq:h-update}, together with the marginal constraints in \eqref{eq:varphi-update} and \eqref{eq:psi-update}, gives
\[
\pi^{n+1,n}\in\mathcal{M}(\mu,\cdot),
\qquad
\pi^{n+1,n+1}\in\Pi(\cdot,\nu).
\]
We set
\[
\nu^{n+1,n}:=\pi^{n+1,n}_{\#y},
\qquad
\mu^{n+1,n+1}:=\pi^{n+1,n+1}_{\#x}.
\]
The martingale half-step is feasible for \eqref{eq:emot} precisely when
\(\nu^{n+1,n}=\nu\).
\end{definition}

Define
\[
\mathcal{X}:=\operatorname{supp}(\mu),
\qquad
\mathcal{Y}:=\operatorname{supp}(\nu).
\]

For a continuous scalar- or vector-valued function \(f\) on a compact set
\(S\), we use the notation
\[
\|f\|_{C(S)}:=\sup_{z\in S}|f(z)|,
\qquad
\operatorname{Lip}(f):=
\sup_{\substack{z,z'\in S\\z\neq z'}}
\frac{|f(z)-f(z')|}{|z-z'|}.
\]
The Lipschitz seminorm is zero when \(S\) is a singleton. We abbreviate
\(\|f\|_{C(S)}\) to \(\|f\|_\infty\) when the domain is clear.

The support condition ensures the solvability of the first equation of the Sinkhorn martingale update.
The strict convex order condition will prevent the sequence of terminal potentials \((\psi^n)_{n\ge 0}\) from becoming
unbounded modulo affine functions. The Lipschitz assumptions on the
cost provide uniform regularity of the updated potentials.

\begin{assumption}[Compact setting, regularity and strict convex order]
\label{ass:two-step}
\leavevmode
\begin{enumerate}
\item \textbf{Geometry.} The sets \(\mathcal X\) and \(\mathcal Y\) are
compact, \(\mathcal Y\) is convex, and
\[
\mathcal X\subset\operatorname{ri}(\mathcal Y),
\qquad \text{where }
\operatorname{ri}(\mathcal Y)
:=
\left\{
y\in\mathcal Y:
\exists r>0\text{ such that }
B(y,r)\cap\operatorname{aff}(\mathcal Y)\subset\mathcal Y
\right\},
\]
and $\operatorname{aff}(\mathcal Y)$ denotes the affine hull of $\mathcal{Y}$.

\item \textbf{Regular reference cost.} The function \(c\) is uniformly
Lipschitz in each variable, with partial Lipschitz seminorms
\[
\begin{aligned}
\operatorname{Lip}_x(c)
&:=\sup_{y\in\mathcal Y}\operatorname{Lip}(c(\cdot,y))<\infty,\\
\operatorname{Lip}_y(c)
&:=\sup_{x\in\mathcal X}\operatorname{Lip}(c(x,\cdot))<\infty.
\end{aligned}
\]

\item \textbf{Strict convex order.} For every continuous convex function
\(u:\mathcal{Y}\to\mathbb{R}\) that is not affine on \(\mathcal{Y}\),
\begin{equation}
\int_\mathcal{X}u(x)\,\mu(\,\mathrm{d} x)
<
\int_\mathcal{Y}u(y)\,\nu(\,\mathrm{d} y).
\label{eq:strict-convex-order}
\end{equation}
\end{enumerate}
\end{assumption}

\begin{remark}[Relative geometry and reduction of dimension]
\label{rem:compact-support-consequences}
The purpose of this remark is to reduce the possibly lower-dimensional
support \(\mathcal Y\) to a full-dimensional setting. Let
\[
A:=\operatorname{aff}(\mathcal Y),
\]
the smallest affine subspace containing \(\mathcal Y\), and let
\(k:=\dim A\) be its affine dimension.

If \(k=0\), both marginals are the same point mass, \(R\) is the point
mass on the corresponding diagonal point, and all coupling iterates
equal \(R\). The main theorems are then immediate, with zero potentials
and any contraction factor in \((0,1)\).

Suppose now that \(k\geq1\), and choose an affine isometry
\(T:\mathbb R^k\to A\). Pulling back the measures and the cost by \(T\)
preserves martingale constraints, relative entropy, strict convex order,
and the Lipschitz seminorms. We keep the same notation for the pulled-back
objects. Then \(\mathcal Y\) has nonempty interior in \(\mathbb R^k\), and
compactness of \(\mathcal X\subset\operatorname{int}(\mathcal Y)\) gives
\(\rho>0\) such that
\begin{equation}
B(x,\rho)\subset\mathcal Y
\qquad \forall x\in\mathcal X.
\label{eq:uniform-interior-ball}
\end{equation}

In all subsequent proofs under Assumption~\ref{ass:two-step}, we work in
these reduced coordinates and write \(d\) for \(k\). Thus \(\mathcal Y\)
is full-dimensional and
\(\operatorname{ri}(\mathcal Y)=\operatorname{int}(\mathcal Y)\).
\end{remark}

\begin{remark}[Strict convex order and positive feasible couplings]
\label{rem:strict-convex-order}
The strict condition in Assumption~\ref{ass:two-step}(3) also implies
ordinary convex order in this compact setting. In the nontrivial case,
\(u_0(y):=|y|^2\) is not affine on \(\mathcal Y\). For every affine
\(\ell\) and every \(t\in\mathbb R\), the function \(u_0+t\ell\)
remains convex and non-affine. Applying the strict inequality to it for
all \(t\) forces
\(\int\ell\,\mathrm d\nu=\int\ell\,\mathrm d\mu\).
The inequality for non-affine convex functions and equality for affine
ones therefore give \(\mu\preceq_{\mathrm{cx}}\nu\). Thus the main
theorems need only the reference-measure part of \cref{ass:main} in
addition to \cref{ass:two-step}.

Under the geometry in Assumption~\ref{ass:two-step}(1), the existence
of \(\gamma\in\mathcal M(\mu,\nu)\) with
\(\gamma\sim\mu\otimes\nu\) implies
Assumption~\ref{ass:two-step}(3). Indeed, disintegration gives
\(\gamma_x\sim\nu\) for \(\mu\)-almost every \(x\).
For a continuous non-affine convex function \(u\) on \(\mathcal Y\)
and such an \(x\in\operatorname{ri}(\mathcal Y)\), choose a supporting
affine function \(\ell_x\) on \(\operatorname{aff}(\mathcal Y)\), with
\(\ell_x\leq u\) on \(\mathcal Y\) and \(\ell_x(x)=u(x)\).
The martingale property gives
\[
\int_\mathcal Y u(y)\,\gamma_x(\mathrm dy)-u(x)
=\int_\mathcal Y (u-\ell_x)(y)\,\gamma_x(\mathrm dy)>0,
\]
since \(u-\ell_x\) is continuous, nonnegative and not identically zero,
while \(\gamma_x\) has full support \(\mathcal Y\). Integrating against
\(\mu\) proves the strict inequality. In particular, the bounded
log-density feasibility assumption in Assumption~2.3 of
\cite{ChenConfortiRenWang2026SinkhornEMOT} implies this condition in the
present geometric setting.
\end{remark}

Our first result bounds the dual potentials associated with the terminal marginal constraint uniformly after
subtraction of suitable affine functions.

\begin{theorem}[Uniform bound modulo affine functions]
\label{thm:uniform-affine-gauge-bound}
Under \cref{ass:main} and \cref{ass:two-step}, there exists
\(C_\psi<\infty\) such that, for every \(n\geq0\), an affine function
\(\ell_n:\mathcal Y\to\mathbb R\) can be chosen with
\[
\|\psi^n-\ell_n\|_{C(\mathcal Y)}\leq C_\psi.
\]
\end{theorem}

This bound yields convergence of the marginals in relative entropy.

\begin{theorem}[Convergence of the marginals of the martingale Sinkhorn iteration]
\label{thm:main-convergence}
Under \cref{ass:main} and \cref{ass:two-step}, we have
\begin{equation}
\begin{aligned}
H\bigl(\mu\mid\mu^{n+1,n+1}\bigr)&\longrightarrow 0,
&
H\bigl(\mu^{n+1,n+1}\mid\mu\bigr)&\longrightarrow 0,
\\
H\bigl(\nu\mid\nu^{n+1,n}\bigr)&\longrightarrow 0,
&
H\bigl(\nu^{n+1,n}\mid\nu\bigr)&\longrightarrow 0,
\end{aligned}
\qquad n\longrightarrow+\infty.
\label{eq:main-convergence}
\end{equation}
\end{theorem}

The uniform potential estimates yield a convergent subsequence of
the martingale half-steps as defined in\cref{def:intermediate-measures}. Marginal
convergence shows that the limit has the prescribed marginals, and its
exponential representation gives optimality and uniqueness.

\begin{theorem}[Representation of the EMOT optimizer]
\label{thm:optimizer_representation}
Under Assumption~\ref{ass:main} and \cref{ass:two-step}, the EMOT problem
\eqref{eq:emot} admits a unique minimizer
$\pi^\star\in\mathcal{M}(\mu,\nu)$ with
\(H(\pi^\star\mid R)<\infty\). Moreover, there exist
\[
\varphi^\star\in\operatorname{Lip}(\mathcal{X}),\qquad
\psi^\star\in\operatorname{Lip}(\mathcal{Y}),\qquad
h^\star\in\operatorname{Lip}(\mathcal{X};\mathbb R^d),
\]
such that
\[
\frac{\mathrm{d}\pi^\star}{\mathrm{d}R}(x,y)
=
\exp\!\left(
-\varphi^\star(x)
-\psi^\star(y)
-h^\star(x)\cdot(y-x)
\right)
\]
for $R$-almost every $(x,y)\in\mathcal{X}\times\mathcal{Y}$.
\end{theorem}

The uniform potential bounds and the stability result of \cref{thm:stability} allow us to prove the exponential convergence of the Sinkhorn algorithm in relative entropy. More precisely, we show that 
\cref{thm:stability} can be applied with 
a constant independent of \(n\), bounding
\(H(\pi^\star\mid\pi^{n+1,n})\) by a fixed multiple of
\(H(\nu\mid\nu^{n+1,n})\). The latter is a term in the full-step
entropy decrement (\cref{prop:half-step-identities}). Hence each
iteration removes at least a fixed fraction of the remaining entropy
error.

\begin{theorem}[Exponential convergence in relative entropy]
\label{thm:exponential-halfstep}
Under Assumption~\ref{ass:main} and \cref{ass:two-step}, let \(\pi^\star\) be the optimizer furnished by \cref{thm:optimizer_representation}. Then there exists \(\alpha\in(0,1)\) such that
\begin{equation}
H\bigl(\pi^\star\mid\pi^{n+1,n}\bigr)
\leq
\alpha^n H\bigl(\pi^\star\mid\pi^{1,0}\bigr),
\qquad n\geq0.
\label{eq:exponential-halfstep}
\end{equation}
\end{theorem}

\section{Entropy identities and optimality}\label{sec:general-results}

In this section we prove the entropy identities for the two
half-steps in \cref{def:intermediate-measures} and an optimality
criterion for couplings with an exponential density representation. These
results use \cref{ass:main} and the finite-step conditions below, but
not the compactness or Lipschitz conditions in \cref{ass:two-step}.

\begin{definition}[Entropy-admissible martingale Sinkhorn iteration]
\label{def:entropy-admissible-iteration}
We call the martingale Sinkhorn iteration in \cref{def:martingale-sinkhorn}
\emph{entropy-admissible} if, for every $n\geq0$, the following properties hold.
\begin{enumerate}
\item There exists a finite Borel measurable function
$h^{n+1}:\mathbb{R}^d\to\mathbb{R}^d$ such that, for $\mu$-almost every $x$,
the integral in \cref{eq:h-update} is absolutely convergent and
\cref{eq:h-update} holds.

\item The integrals on the right-hand sides of
\cref{eq:varphi-update,eq:psi-update} are finite and strictly positive for
$\mu$-almost every $x$ and $\nu$-almost every $y$, respectively, and the
resulting functions $\varphi^{n+1}$ and $\psi^{n+1}$ admit finite Borel
measurable representatives.

\item At every finite step,
\[
\varphi^n\in L^\infty(\mu),
\qquad
\psi^n\in L^\infty(\nu),
\qquad
h^n\in L^\infty(\mu;\mathbb{R}^d).
\]
These bounds may depend on $n$.
\end{enumerate}
\end{definition}

For an entropy-admissible iteration, multiplying the integral in \cref{eq:h-update}
by a strictly positive function of
\(x\) does not change its zero set. Thus replacing \(\varphi^n\) by
\(\varphi^{n+1}\) in \cref{eq:h-update} is harmless, and the
martingale half-step of \cref{def:intermediate-measures} belongs to
\(\mathcal M(\mu,\cdot)\).
The update of \cref{eq:psi-update} and Tonelli's theorem give
\begin{equation}
\frac{\mathrm d\nu^{n+1,n}}{\mathrm d\nu}(y)
=e^{\psi^{n+1}(y)-\psi^n(y)}
\qquad\text{for }\nu\text{-almost every }y.
\label{eq:intermediate-marginal-density}
\end{equation}
Consequently both marginals of \(\pi^{n+1,n}\) have finite first moments.
The bounded density
\(\mathrm d\pi^{n+1,n+1}/\mathrm d\pi^{n+1,n}
=e^{\psi^n-\psi^{n+1}}\)
gives the same conclusion for the terminal-marginal half-step of
\cref{def:intermediate-measures}.

The relevant logarithmic ratios are
\begin{align}
\log\frac{\mathrm d\pi^{n+1,n}}{\mathrm d\pi^{n,n}}(x,y)
&=\varphi^n(x)-\varphi^{n+1}(x)
 +(h^n(x)-h^{n+1}(x))\cdot(y-x),
\label{eq:first-halfstep-ratio}\\
\log\frac{\mathrm d\pi^{n+1,n+1}}{\mathrm d\pi^{n+1,n}}(x,y)
&=\psi^n(y)-\psi^{n+1}(y).
\label{eq:second-halfstep-ratio}
\end{align}
The first is integrable under every martingale measure with first
marginal \(\mu\) and finite first moments; the second is bounded.
If \(\bar\pi\in\mathcal M(\mu,\nu)\) has
\(H(\bar\pi\mid R)<\infty\), the logarithm of each iterate's density
relative to \(R\) is \(\bar\pi\)-integrable. Its entropy relative to
every finite iterate is therefore finite as well.

\begin{proposition}[Half-step entropy identities and full-step decrement]
\label{prop:half-step-identities}
Suppose the martingale Sinkhorn iteration is entropy-admissible. Then, for every \(n\geq0\),
\(\pi^{n+1,n}\) and \(\pi^{n+1,n+1}\) are the unique minimizers of
\(H(\cdot\mid\pi^{n,n})\) over \(\mathcal M(\mu,\cdot)\), and of
\(H(\cdot\mid\pi^{n+1,n})\) over \(\Pi(\cdot,\nu)\), respectively.
If \(\bar\pi\in\mathcal M(\mu,\nu)\) satisfies
\(H(\bar\pi\mid R)<\infty\), then
\begin{align}
H(\bar\pi\mid\pi^{n,n})
&=H(\bar\pi\mid\pi^{n+1,n})
 +H(\pi^{n+1,n}\mid\pi^{n,n}),
\label{eq:pythagorean-first}\\
H(\bar\pi\mid\pi^{n+1,n})
&=H(\bar\pi\mid\pi^{n+1,n+1})
 +H(\pi^{n+1,n+1}\mid\pi^{n+1,n}).
\label{eq:pythagorean-second}
\end{align}
In particular,
\begin{align}
&H(\bar\pi\mid\pi^{n+1,n})-H(\bar\pi\mid\pi^{n+2,n+1})\notag\\
&\qquad=H(\pi^{n+2,n+1}\mid\pi^{n+1,n+1})
 +H(\pi^{n+1,n+1}\mid\pi^{n+1,n}).
\label{eq:full-step-decrement}
\end{align}
The second term has the exact marginal representation
\begin{equation}
H(\pi^{n+1,n+1}\mid\pi^{n+1,n})
=H(\nu\mid\nu^{n+1,n}).
\label{eq:exact-marginal-decrement}
\end{equation}
The marginal errors also satisfy
\begin{equation}
\begin{aligned}
\max\bigl\{H(\mu\mid\mu^{n+1,n+1}),\,
H(\nu^{n+2,n+1}\mid\nu)\bigr\}
&\leq H(\pi^{n+2,n+1}\mid\pi^{n+1,n+1}),\\
\max\bigl\{H(\mu^{n+1,n+1}\mid\mu),\,
H(\nu\mid\nu^{n+1,n})\bigr\}
&\leq H(\pi^{n+1,n+1}\mid\pi^{n+1,n}).
\end{aligned}
\label{eq:marginal-data-processing}
\end{equation}
\end{proposition}

\begin{proof}
Let \(\gamma\in\mathcal M(\mu,\cdot)\) with
\(H(\gamma\mid\pi^{n,n})<\infty\).
By \cref{eq:first-halfstep-ratio}, the entropy chain rule and the
martingale constraint give
\[
H(\gamma\mid\pi^{n,n})
=H(\gamma\mid\pi^{n+1,n})
 +\int(\varphi^n-\varphi^{n+1})\,\mathrm d\mu.
\]
The last integral equals
\(H(\pi^{n+1,n}\mid\pi^{n,n})<\infty\), as seen by evaluating the
same logarithmic ratio under \(\pi^{n+1,n}\).
For \(\gamma\in\Pi(\cdot,\nu)\) with \(H(\gamma\mid\pi^{n+1,n})<\infty\), the bounded
ratio in \cref{eq:second-halfstep-ratio} similarly gives
\[
H(\gamma\mid\pi^{n+1,n})
=H(\gamma\mid\pi^{n+1,n+1})
 +\int(\psi^n-\psi^{n+1})\,\mathrm d\nu,
\]
where the last integral equals
\(H(\pi^{n+1,n+1}\mid\pi^{n+1,n})\).
The nonnegativity of relative entropy, with equality only for identical
probability measures, gives both minimization claims and uniqueness.

Taking \(\gamma=\bar\pi\) gives
\cref{eq:pythagorean-first,eq:pythagorean-second}.
Adding the second identity at time \(n\) and the first at time \(n+1\)
proves \cref{eq:full-step-decrement}. Finally,
\cref{eq:intermediate-marginal-density,eq:second-halfstep-ratio} show that
\[
\log\frac{\mathrm d\pi^{n+1,n+1}}{\mathrm d\pi^{n+1,n}}(x,y)
=\log\frac{\mathrm d\nu}{\mathrm d\nu^{n+1,n}}(y).
\]
Integration under the coupling with second marginal \(\nu\) proves
\cref{eq:exact-marginal-decrement}. The bounds in
\cref{eq:marginal-data-processing} follow by applying data processing inequality
to the coordinate projections.
\end{proof}

For a martingale coupling with exponential representation, the integral of its
log-density is the same under every coupling satisfying the marginal and martingale
constraints. This property, combined with the entropy chain rule, allows to obtain an optimality criterion.

\begin{theorem}[Pythagorean identity for EMOT]
\label{thm:general-pythagorean}
Let $\pi^\star\in\mathcal{M}(\mu,\nu)$ and
$R\in\mathcal{P}(\mathbb{R}^d\times\mathbb{R}^d)$. Assume there are finite Borel representatives with
\[
\varphi\in L^1(\mu),
\qquad
\psi\in L^1(\nu),
\qquad
h\in L^\infty(\mu;\mathbb{R}^d)
\]
such that
\begin{equation}
\frac{\mathrm{d}\pi^\star}{\mathrm{d}R}(x,y)
=
\exp\bigl(-\varphi(x)-\psi(y)-h(x)\cdot(y-x)\bigr).
\label{eq:general-exponential-representation}
\end{equation}
Then $H(\pi^\star\mid R)<+\infty$
and for every $\pi\in\mathcal{M}(\mu,\nu)$ with
$H(\pi\mid R)<+\infty$,
\begin{equation}
H(\pi\mid R)
=
H\bigl(\pi\mid\pi^\star\bigr)
+
H\bigl(\pi^\star\mid R\bigr).
\label{eq:general-pythagorean}
\end{equation}
In particular, \(\pi^\star\) is the unique minimizer of the EMOT
problem \eqref{eq:emot}.
\end{theorem}

\begin{proof}
For every \(\gamma\in\mathcal M(\mu,\nu)\),
\[
\int\!\bigl|\varphi(x)+\psi(y)+h(x)\cdot(y-x)\bigr|\,\mathrm d\gamma
\leq \|\varphi\|_{L^1(\mu)}+\|\psi\|_{L^1(\nu)}
+\|h\|_{L^\infty(\mu)}
\left(\int |x|\,\mathrm d\mu+\int |y|\,\mathrm d\nu\right)<\infty.
\]
The marginal and martingale constraints therefore give
\[
\int\log\frac{\mathrm d\pi^\star}{\mathrm dR}\,\mathrm d\gamma
=-\int\varphi\,\mathrm d\mu-\int\psi\,\mathrm d\nu.
\]
Taking \(\gamma=\pi^\star\) shows that the right-hand side is
\(H(\pi^\star\mid R)<\infty\).

The exponential density is strictly positive, so \(\pi^\star\sim R\).
For \(\pi\in\mathcal M(\mu,\nu)\) with \(H(\pi\mid R)<\infty\),
the same integrability estimate justifies the entropy chain rule:
\[
H(\pi\mid R)
=H(\pi\mid\pi^\star)
 +\int\log\frac{\mathrm d\pi^\star}{\mathrm dR}\,\mathrm d\pi
=H(\pi\mid\pi^\star)+H(\pi^\star\mid R).
\]
Nonnegativity of relative entropy proves optimality. Any other
minimizer has finite entropy and satisfies \(H(\pi\mid\pi^\star)=0\),
hence equals \(\pi^\star\).
\end{proof}

For the martingale half-step \(\pi^{n+1,n}\) of
\cref{def:intermediate-measures}, the first marginal and the martingale
constraints are satisfied. The same optimality criterion therefore applies
with \(\nu\) replaced by \(\nu^{n+1,n}\).

\begin{corollary}[Optimality of the intermediate martingale coupling]
\label{cor:intermediate-optimality}
Under \cref{ass:main}, suppose that the martingale Sinkhorn iteration is
entropy-admissible in the sense of
\cref{def:entropy-admissible-iteration}. Then, for every $n\geq 0$,
$\pi^{n+1,n}$ is the optimal martingale coupling between
$\mu$ and $\nu^{n+1,n}$ under the reference measure $R$.
\end{corollary}

\begin{proof}
The intermediate coupling has finite first moments, as shown after
\cref{def:entropy-admissible-iteration}, and its logarithmic density relative
to \(R\) is
\(-\varphi^{n+1}(x)-\psi^n(y)-h^{n+1}(x)\cdot(y-x)\).
The entropy-admissibility of the martingale Sinkhorn iteration ensures that the potentials satisfy the conditions of
\cref{thm:general-pythagorean}, with terminal marginal \(\nu^{n+1,n}\).
\cref{thm:general-pythagorean} then gives optimality and uniqueness.
\end{proof}

\section{Uniform estimates and existence of the optimizer}\label{sec:two-step-estimates}

The main estimate of this section is a uniform, modulo affine gauges, bound on the dual potentials associated with the terminal marginal constraint. We obtain it from the strict convex order condition and
the upper bound on a decreasing functional. The bound on the potentials also
gives a lower bound on the same functional, which proves marginal
convergence. Finally, regularity of the updates lets us pass to a limit
and obtain the optimizer.

\subsection{Well-posedness and affine gauges}

In this section we will establish bounds on the dual potentials using their integral representations. We will sometimes restrict the integrals with respect to the measure $\nu(dy)$ to the sets of the form $ B(y, r) $ and  \(\{y:e\cdot(y-x)\leq \tilde r\}\) for given $r, \tilde r, e, x$. We want the $\nu$-measure of these sets to be strictly positive to obtain lower bounds. For that purpose choose
 \(\rho>0\) as in \cref{eq:uniform-interior-ball}.
By \cref{lem:compact-support-nondegeneracy},
\begin{equation}
m_\nu(r):=\inf_{y\in\mathcal Y}\nu(B(y,r))>0
\qquad\text{for every fixed }r>0.
\label{eq:small-ball-mass}
\end{equation}
For every \(x\in\mathcal X\) and \(e\in\mathbb S^{d-1}\), the ball
\(B(x-\rho e/2,\rho/4)\) is contained in \(\mathcal Y\) and in
\(\{y:e\cdot(y-x)\leq-\rho/4\}\). Its center belongs to
\(\mathcal Y=\operatorname{supp}\nu\), so
\begin{equation}
\label{eq:nu_set_lower_bound}
\nu\bigl(\{y\in\mathcal Y:e\cdot(y-x)\leq-\rho/4\}\bigr)\geq m_\nu(\rho/4).
\end{equation}

The estimates in this section will be stated for a general
potential $\psi$, not necessarily coming from the martingale Sinkhorn iteration. We therefore introduce Sinkhorn operators, allowing us to write one
complete update as \(\psi\mapsto(h_\psi,\varphi_\psi,\mathcal S\psi)\).

\begin{definition}[Martingale Sinkhorn operators]
\label{def:operator-martingale-sinkhorn}
For every function $\psi:\mathbb{R}^d\to\mathbb{R}$ for which the following construction is
well defined, let $h_\psi:\mathbb{R}^d\to\mathbb{R}^d$ be specified by
\begin{equation}
0
=
\int_{\mathbb{R}^d}
(y-x)
\exp\bigl(
-\psi(y)
-h_\psi(x)\cdot(y-x)
-c(x,y)
\bigr)
\,\nu(\,\mathrm{d} y)
\qquad\text{for }\mu\text{-almost every }x\in\mathbb{R}^d.
\label{eq:operator-h-update}
\end{equation}
When $h_{\psi}$ is not unique, fix a measurable choice.
Under Assumption~\ref{ass:two-step}, uniqueness follows from
\cref{prop:compact-well-posedness} below, after the reduction of
\cref{rem:compact-support-consequences}.

Next, define $\varphi_\psi:\mathbb{R}^d\to\mathbb{R}$ by
\begin{equation}
\varphi_\psi(x)
:=
\log\!\left(
\int_{\mathbb{R}^d}
\exp\bigl(
-\psi(y)
-h_\psi(x)\cdot(y-x)
-c(x,y)
\bigr)
\,\nu(\,\mathrm{d} y)
\right).
\label{eq:operator-varphi-update}
\end{equation}

Finally, define the operator $\mathcal S$ by
\begin{equation}
(\mathcal S\psi)(y)
:=
\log\!\left(
\int_{\mathbb{R}^d}
\exp\bigl(
-\varphi_\psi(x)
-h_\psi(x)\cdot(y-x)
-c(x,y)
\bigr)
\,\mu(\,\mathrm{d} x)
\right).
\label{eq:operator-psi-update}
\end{equation}

With the initialization in \cref{def:martingale-sinkhorn}, the martingale Sinkhorn iteration
takes the form
\[
h^{n+1}=h_{\psi^n},\qquad
\varphi^{n+1}=\varphi_{\psi^n},\qquad
\psi^{n+1}=\mathcal S\psi^n,\qquad n\geq0.
\]
\end{definition}

We now show that for fixed \(x\), \cref{eq:operator-h-update} is the first-order optimality condition
for a strictly convex functional. The support condition makes this
functional coercive and gives a bound on the functional that is linear in
\(\|\psi\|_{C(\mathcal Y)}\). This linear bound will be used to control
the sequence of Sinkhorn iterates.

\begin{proposition}[Well-posedness of the iteration]
\label{prop:compact-well-posedness}
Under \cref{ass:main} and \cref{ass:two-step}, for every
$\psi\in C(\mathcal{Y})$ and every $x\in\mathcal{X}$, the martingale equation
\begin{equation*}
0
=
\int_\mathcal{Y}(y-x)
\exp\bigl(-\psi(y)-h_\psi(x)\cdot(y-x)-c(x,y)\bigr)
\,\nu(\,\mathrm{d}y)
\end{equation*}
admits a unique solution $h_\psi(x)\in\mathbb{R}^d$. Moreover,
\[
h_\psi\in C(\mathcal{X};\mathbb{R}^d),
\qquad
\varphi_\psi\in C(\mathcal{X}),
\qquad
\mathcal{S}\psi\in C(\mathcal{Y}),
\]
and all three functions are bounded on their respective domains. In addition,
there is \(C_h>0\), depending only on \(\rho\), \(m_\nu(\rho/4)\), and
\(\|c\|_{C(\mathcal X\times\mathcal Y)}\), such that
\begin{equation}
\|h_\psi\|_{C(\mathcal X)}
\leq C_h\bigl(1+\|\psi\|_{C(\mathcal Y)}\bigr).
\label{eq:linear-multiplier-bound}
\end{equation}
Then the martingale Sinkhorn
iteration is entropy-admissible.
\end{proposition}

\begin{proof}
Fix $\psi\in C(\mathcal{Y})$ and set
\[
M:=\lVert\psi\rVert_{C(\mathcal{Y})},
\qquad
C_c:=\lVert c\rVert_{C(\mathcal{X}\times\mathcal{Y})}.
\]
For $x\in\mathcal{X}$ and $p\in\mathbb{R}^d$, define
\begin{equation}
\label{eq:F_psi_functional}
F_\psi(x,p)
:=
\log\!\left(
\int_\mathcal{Y}
\exp\bigl(-\psi(y)-p\cdot(y-x)-c(x,y)\bigr)
\,\nu(\,\mathrm{d}y)
\right).
\end{equation}
The function $F_\psi$ is jointly continuous. Since $\mathcal{Y}$ is compact,
differentiation with respect to $p$ gives
\[
\nabla_pF_\psi(x,p)
=
x-\int_\mathcal{Y}y\,Q_{x,p}(\,\mathrm{d}y),
\qquad
\nabla_{pp}^2F_\psi(x,p)
=
\operatorname{Cov}_{Q_{x,p}}(Y),
\]
where
\[
Q_{x,p}(\,\mathrm{d}y)
:=
\frac{
\exp\bigl(-\psi(y)-p\cdot(y-x)-c(x,y)\bigr)
}{
\displaystyle
\int_\mathcal{Y}
\exp\bigl(-\psi(z)-p\cdot(z-x)-c(x,z)\bigr)
\,\nu(\,\mathrm{d}z)
}
\,\nu(\,\mathrm{d}y).
\]

For $p\neq0$, let $e=p/\lvert p\rvert$. 
Restricting the integral in \cref{eq:F_psi_functional} to the set
$-p\cdot(y-x)\geq \frac{\rho}{4}\lvert p\rvert$ and using \cref{eq:nu_set_lower_bound}, we obtain
\begin{equation*}
F_\psi(x,p)
\geq
-M-C_c+\frac{\rho}{4}\lvert p\rvert+\log m_\nu(\rho/4).
\end{equation*}
Thus $F_\psi(x,\cdot)$ is coercive and attains a minimum. The measure
$Q_{x,p}$ is equivalent to $\nu$, hence has support $\mathcal Y$.
Since \(\mathcal Y\) is full-dimensional, no nonzero linear functional
is constant \(Q_{x,p}\)-almost surely. Its covariance matrix is therefore
positive definite. Hence
$F_\psi(x,\cdot)$ is strictly convex and its minimizer is unique. Denote the
minimizer by $h_\psi(x)$. Its
first-order optimality condition is exactly \cref{eq:operator-h-update}.

The minimizer is bounded uniformly in $x$. Indeed,
\[
F_\psi\bigl(x,h_\psi(x)\bigr)
\leq
F_\psi(x,0)
\leq
M+C_c,
\]
which, together with the preceding lower bound, gives
\begin{equation}
\lvert h_\psi(x)\rvert
\leq
\frac4\rho\bigl(2M+2C_c-\log m_\nu(\rho/4)\bigr)
\qquad
\text{for every }x\in\mathcal{X}.
\label{eq:finite-step-h-bound}
\end{equation}
Increasing a constant depending only on the fixed quantities in this
bound gives \cref{eq:linear-multiplier-bound}.

We next prove continuity. Let $x_k\to x$ in $\mathcal{X}$. By
\cref{eq:finite-step-h-bound}, every subsequence of
$\bigl(h_\psi(x_k)\bigr)_k$ has a further convergent subsequence. If its limit
is $p$, the minimizing property and the joint continuity of $F_\psi$ imply
\[
F_\psi(x,p)\leq F_\psi(x,q)
\qquad
\text{for every }q\in\mathbb{R}^d.
\]
By uniqueness, $p=h_\psi(x)$. Thus the whole sequence converges and
$h_\psi$ is continuous.

The identity
\[
\varphi_\psi(x)=F_\psi\bigl(x,h_\psi(x)\bigr)
\]
shows that $\varphi_\psi$ is continuous and bounded on $\mathcal{X}$.
The integrand in \cref{eq:operator-psi-update} is then positive and continuous
on the compact set $\mathcal{X}\times\mathcal{Y}$, so
$\mathcal{S}\psi$ is continuous and bounded on $\mathcal{Y}$.

Starting from $\psi^0=0$ and arguing inductively, every finite iteration
therefore produces bounded continuous dual variables. All update integrals are
finite, the martingale integral is absolutely convergent, and the intermediate
measures are supported on the compact set
$\mathcal{X}\times\mathcal{Y}$. The martingale and marginal constraints give the two one-sided constraints
in \cref{def:intermediate-measures}. Hence the iteration is
entropy-admissible.
\end{proof}

Direct uniform bounds on the dual potentials are not available. We therefore introduce affine gauges and later show that the potentials are uniformly bounded after a suitable gauge transformation.

\begin{definition}[Affine gauges]
\label{def:affine-gauges}
Let
\[
\mathcal{A}(\mathcal{Y})
:=
\left\{
\ell:\mathcal{Y}\to\mathbb{R}:\ell(y)=a+b\cdot y,\ a\in\mathbb{R},\ b\in\mathbb{R}^d
\right\}.
\]
For \(\ell(y)=a+b\cdot y\in\mathcal{A}(\mathcal Y)\) and \(k\geq0\), define
\[
\varphi^{k,\ell}:=\varphi^k+\ell|_{\mathcal X},\qquad
\psi^{k,\ell}:=\psi^k-\ell,\qquad
h^{k,\ell}:=h^k+b.
\]
The first superscript records the iteration level and the second the
chosen gauge. For all \(k,m\geq0\) and \((x,y)\in\mathcal X\times\mathcal Y\),
\begin{equation}
\varphi^{k,\ell}(x)+\psi^{m,\ell}(y)+h^{k,\ell}(x)\cdot(y-x)
=
\varphi^k(x)+\psi^m(y)+h^k(x)\cdot(y-x).
\label{eq:affine-gauge-invariance}
\end{equation}
Thus the intermediate measures (half-steps) are unchanged under this gauge transformation.
\end{definition}

The same affine change can be used throughout the martingale Sinkhorn iteration. Thus we may
estimate the next iterate in a gauge selected for the current potential $\psi$.

\begin{proposition}[Affine invariance of the Sinkhorn updates]
\label{prop:gauge-preserves-sinkhorn}
Under \cref{ass:main} and \cref{ass:two-step}, let
\(\ell(y)=a+b\cdot y\in\mathcal A(\mathcal Y)\).
If \((\varphi^n,\psi^n,h^n)_{n \geq 0}\) satisfy the Sinkhorn update
\cref{eq:h-update,eq:varphi-update,eq:psi-update}, then their affine gauge
transforms
\[
\varphi^{n,\ell}=\varphi^n+\ell|_{\mathcal X},\qquad
\psi^{n,\ell}=\psi^n-\ell,\qquad
h^{n,\ell}=h^n+b
\]
satisfy the same update equations for $n \geq 0$.
\end{proposition}

\begin{proof}
By \cref{eq:affine-gauge-invariance} the exponential terms in the three Sinkhorn updates
\cref{eq:h-update,eq:varphi-update,eq:psi-update} are unchanged.
Therefore the gauge-transformed potentials satisfy the same update equations.
\end{proof}

\subsection{Uniform bounds and marginal convergence}

We next introduce a functional defined entirely from the iterates.
Its decrease is the sum of the two half-step entropies from
\cref{prop:half-step-identities}. It allows us
to study convergence before a finite-entropy feasible coupling has
been constructed.

\begin{proposition}[Monotonicity of the gauge-invariant functional]
\label{prop:reduced-lyapunov}
Under \cref{ass:main} and \cref{ass:two-step}, define
\[
\Phi_n:=\int_\mathcal X\varphi^{n+1}\,\mathrm d\mu
 +\int_\mathcal Y\psi^n\,\mathrm d\nu.
\]
Then \(\Phi_n\) is invariant under affine gauges and
\begin{align}
\Phi_n-\Phi_{n+1}
&=H(\pi^{n+2,n+1}\mid\pi^{n+1,n+1})\notag\\
&\quad+H(\pi^{n+1,n+1}\mid\pi^{n+1,n})\geq0.
\label{eq:Phi-decrement}
\end{align}
Moreover, \(\Phi_n\leq\Phi_0\leq0\) for every \(n\geq0\).
\end{proposition}

\begin{proof}
The equality of barycenters in \cref{ass:main} gives
\(\int\ell\,\mathrm d\mu=\int\ell\,\mathrm d\nu\) for every affine
\(\ell\), which proves gauge invariance.
All finite-step potentials are bounded by
\cref{prop:compact-well-posedness}. Computing the logarithmic density
ratios and using the one-sided constraints gives
\begin{align*}
H(\pi^{n+2,n+1}\mid\pi^{n+1,n+1})
&=\int_\mathcal X(\varphi^{n+1}-\varphi^{n+2})\,\mathrm d\mu,\\
H(\pi^{n+1,n+1}\mid\pi^{n+1,n})
&=\int_\mathcal Y(\psi^n-\psi^{n+1})\,\mathrm d\nu.
\end{align*}
Adding these identities proves \cref{eq:Phi-decrement}.
Since \(\psi^0=0\), the pointwise minimizing property established in
\cref{prop:compact-well-posedness}, with competitor \(p=0\), gives
\[
\varphi^1(x)
=
F_0(x,h^1(x))
=
\min_{p\in\mathbb{R}^d} F_0(x,p)
\le
F_0(x,0)
=
\log \int_Y e^{-c(x,y)}\,\nu(dy).
\]
By Jensen's inequality, we have
\[
\Phi_0=\int_\mathcal X\varphi^1\,\mathrm d\mu
\leq\log\int_{\mathcal X\times\mathcal Y}e^{-c(x,y)}
 \,\mu(\mathrm dx)\nu(\mathrm dy).
\]
Since \(R\) is a probability measure and \(\frac{dR}{d(\mu\otimes\nu)}(x,y)=e^{-c(x,y)}\), we have
\[
\int_{X\times Y}e^{-c(x,y)}\,\mu(dx)\nu(dy)=R(X\times Y)=1.
\]
Hence \(\Phi_0\le 0\). By monotonicity, \(\Phi_n\le\Phi_0\le0\) for every \(n\ge0\).
\end{proof}

We now proceed to the proof of \cref{thm:uniform-affine-gauge-bound}. We show that the potential $\psi^n$ differs by a uniformly bounded function from a
convex log-partition function. If the potentials $\psi^n$ were
unbounded modulo affine functions, rescaling them would therefore
give a non-affine convex limit. Strict convex order condition would then
contradict the upper bound on \(\Phi_n\).

\begin{proof}[Proof of \cref{thm:uniform-affine-gauge-bound}]
Set
\[
T_n:=\inf_{\ell\in\mathcal A(\mathcal Y)}
\|\psi^n-\ell\|_{C(\mathcal Y)}.
\]
Suppose that \((T_n)_{n\geq0}\) is unbounded. Choose record times
\(n_j\geq1\) such that
\[
t_j:=T_{n_j}=\max_{0\leq m\leq n_j}T_m\longrightarrow\infty,
\qquad t_j\geq1.
\]
Since \(T_{n_j-1}\leq t_j\), we may choose
\(\ell_j\in\mathcal A(\mathcal Y)\) with
\[
\|\psi^{n_j-1,\ell_j}\|_{C(\mathcal Y)}\leq2t_j.
\]
We keep this gauge for both of the next two updates.

\emph{Compactness after rescaling.}
By \cref{prop:compact-well-posedness,prop:gauge-preserves-sinkhorn}
and the normalization updates,
\[
\|h^{n_j,\ell_j}\|_{C(\mathcal X)}
+\|\varphi^{n_j,\ell_j}\|_{C(\mathcal X)}
+\|\psi^{n_j,\ell_j}\|_{C(\mathcal Y)}
\leq C(1+t_j),
\]
where \(C\) is independent of \(j\). Comparing the terminal update at
two points also gives
\[
\operatorname{Lip}(\psi^{n_j,\ell_j})
\leq\|h^{n_j,\ell_j}\|_{C(\mathcal X)}
 +\operatorname{Lip}_y(c)\leq C(1+t_j).
\]
Define \(u_j:=\psi^{n_j,\ell_j}/t_j\). Since \(t_j\geq1\),
\[
\|u_j\|_{C(\mathcal Y)}+\operatorname{Lip}(u_j)\leq C,
\qquad
\inf_{\ell\in\mathcal A(\mathcal Y)}
\|u_j-\ell\|_{C(\mathcal Y)}=1.
\]
The last equality holds for any affine gauge: subtracting \(\ell_j\)
does not change the distance to \(\mathcal A(\mathcal Y)\), and
rescaling divides that distance by \(t_j=T_{n_j}\).
Arzel\`a--Ascoli theorem gives a subsequence converging uniformly on
\(\mathcal Y\) to a continuous function \(u\).

\emph{Convexity of the limit.}
Put \(C_c:=\|c\|_{C(\mathcal X\times\mathcal Y)}\) and define
\[
G_j(y):=\log\int_\mathcal X
\exp\bigl(-\varphi^{n_j,\ell_j}(x)
-h^{n_j,\ell_j}(x)\cdot(y-x)\bigr)\,\mu(\mathrm dx).
\]
This is a finite convex function of \(y\), by H\"older's inequality.
The $\psi$-update \cref{eq:psi-update} gives
\[
\|\psi^{n_j,\ell_j}-G_j\|_{C(\mathcal Y)}\leq C_c,
\qquad
\|u_j-G_j/t_j\|_{C(\mathcal Y)}\leq C_c/t_j\longrightarrow0.
\]
Hence \(u\) is convex. Distance to \(\mathcal A(\mathcal Y)\) is
\(1\)-Lipschitz in the supremum norm, so
\[
\inf_{\ell\in\mathcal A(\mathcal Y)}
\|u-\ell\|_{C(\mathcal Y)}=1.
\]
In particular, \(u\) is not affine.

\emph{Use of strict convex order.}
Apply the next update in the same gauge \(\ell_j\). The bound on
\(u_j\) and \cref{eq:linear-multiplier-bound} imply
\[
\|h^{n_j+1,\ell_j}\|_{C(\mathcal X)}\leq C t_j.
\]
Fix \(0<\varepsilon<\rho\). For  \(y \in B(x,\varepsilon)\subset\mathcal Y\), we have
\[
u_j(y)\leq u_j(x)+\operatorname{Lip}(u_j)\varepsilon,
\qquad
|h^{n_j+1,\ell_j}(x)\cdot(y-x)|\leq C t_j\varepsilon.
\]
Restricting the $\varphi$-update \cref{eq:varphi-update} to this ball gives
\[
\varphi^{n_j+1,\ell_j}(x)
\geq-t_ju_j(x)-Ct_j\varepsilon-C_c+\log m_\nu(\varepsilon).
\]
Here \(C\) is independent of both \(j\) and \(\varepsilon\).
After division by \(t_j\) and integration against \(\mu\), let
\(j\to\infty\) for fixed \(\varepsilon\), and then let
\(\varepsilon \to 0\). Since \(m_\nu(\varepsilon)>0\), this yields
\[
\liminf_{j\to\infty}\frac1{t_j}
\int_\mathcal X\varphi^{n_j+1,\ell_j}\,\mathrm d\mu
\geq-\int_\mathcal X u\,\mathrm d\mu.
\]
Gauge invariance and uniform convergence of \(u_j\) now imply
\[
\liminf_{j\to\infty}\frac{\Phi_{n_j}}{t_j}
\geq\int_\mathcal Y u\,\mathrm d\nu
 -\int_\mathcal X u\,\mathrm d\mu>0,
\]
where the strict inequality follows from
\cref{eq:strict-convex-order}. This contradicts
\(\Phi_{n_j}\leq\Phi_0\leq0\) from
\cref{prop:reduced-lyapunov}. Thus \((T_n)_{n\geq0}\) is bounded.
Taking \(C_\psi:=1+\sup_{n\geq0}T_n\) proves the claimed result.
\end{proof}

We now choose one bounded representative of $\psi^n$ at each step of the iteration. Later
estimates use the same chosen gauge for consecutive updates, so the
iteration index and the gauge index must be kept separate.

\begin{definition}[Good gauges]
\label{def:good-gauges}
Under \cref{ass:main} and \cref{ass:two-step}, fix
\(C_\psi\) as in that theorem. For every \(n\geq0\), choose an affine
function \(\ell_n\in\mathcal A(\mathcal Y)\) such that
\[
\lVert\psi^n-\ell_n\rVert_{C(\mathcal Y)}\leq C_\psi.
\]
We call such a choice a \emph{good gauge} at level \(n\). We take
\(\ell_0=0\), since \(\psi^0=0\), and fix one choice for each \(n\geq1\).
For these fixed gauges, set
\[
(\varphi^{k,[n]},\psi^{k,[n]},h^{k,[n]})
:=(\varphi^{k,\ell_n},\psi^{k,\ell_n},h^{k,\ell_n}),
\qquad k\geq0.
\]
The label \([n]\) denotes the fixed gauge \(\ell_n\), independently of
the iteration level \(k\). We use the explicit gauge superscript when
it is necessary to distinguish gauges.
\end{definition}

The uniform gauge bound also gives a lower bound on \(\Phi_n\).
We can therefore sum its decrement identity \eqref{eq:Phi-decrement} and deduce convergence
of the marginals.

\begin{proof}[Proof of \cref{thm:main-convergence}]
Fix \(n\geq0\) and use the good gauge \(\ell_n\) from
\cref{def:good-gauges}.
By \cref{eq:linear-multiplier-bound,prop:gauge-preserves-sinkhorn},
\[
\|\psi^{n,[n]}\|_{C(\mathcal Y)}\leq C_\psi,
\qquad
\|h^{n+1,[n]}\|_{C(\mathcal X)}\leq C_h(1+C_\psi).
\]
Then the $\varphi$-update \cref{eq:varphi-update} gives
\[
\|\varphi^{n+1,[n]}\|_{C(\mathcal X)}
\leq C_\psi+\operatorname{diam}(\mathcal Y) C_h(1+C_\psi)+\|c\|_{C(\mathcal X\times\mathcal Y)}.
\]
Gauge invariance implies \(\inf_n\Phi_n> -\infty\).
Telescoping \cref{eq:Phi-decrement} now shows that
\[
\sum_{n=0}^\infty
\bigl[H(\pi^{n+2,n+1}\mid\pi^{n+1,n+1})
 +H(\pi^{n+1,n+1}\mid\pi^{n+1,n})\bigr]<\infty.
\]
Both summands tend to zero. Applying
\cref{eq:marginal-data-processing}, with an index shift, proves all
four limits in \cref{eq:main-convergence}.
\end{proof}

\subsection{Regularity and existence of the optimizer}

A bound on the potential $\psi$ gives uniform regularity of $\mathcal{S}\psi$, with the Lipschitz constant depending on the supremum norm of $\psi$, but not on its
Lipschitz seminorm. We use this fact to extract a limit while retaining
the exponential form of the density.

Under \cref{ass:two-step} by \cref{lem:compact-support-nondegeneracy}, with
\(\bar y_\nu:=\int_\mathcal Y y\,\nu(\mathrm dy)\), we have that
\begin{equation}
\sigma_\nu:=\lambda_{\min}\!\left(
\int_\mathcal Y(y-\bar y_\nu)(y-\bar y_\nu)^\top\,\nu(\mathrm dy)
\right)>0.
\label{eq:reference-covariance}
\end{equation}
This covariance bound controls the dependence of the potential $h$ on
\(x\). The other two potentials can be estimated directly from the
update formulas.

\begin{proposition}[Bounds for one Sinkhorn update]
\label{prop:multiplier-lipschitz}
Let \cref{ass:main} and \cref{ass:two-step} hold. Then for every \(0\leq M<\infty\),
there is a constant \(C_M<\infty\), depending only on \(M\), the
compact supports, the cost bounds, \(m_\nu(\rho/4)\), and
\(\sigma_\nu\), such that for every \(\psi\in C(\mathcal Y)\) with
\(\|\psi\|_{C(\mathcal Y)}\leq M\) we have
\[
\begin{aligned}
&\|h_\psi\|_{C(\mathcal X)}+\operatorname{Lip}(h_\psi)
 +\|\varphi_\psi\|_{C(\mathcal X)}+\operatorname{Lip}(\varphi_\psi)\\
&\qquad+\|\mathcal S\psi\|_{C(\mathcal Y)}
 +\operatorname{Lip}(\mathcal S\psi)\leq C_M.
\end{aligned}
\]
In particular,
\begin{equation}
\begin{aligned}
\operatorname{Lip}(\varphi_\psi)
&\leq\|h_\psi\|_{C(\mathcal X)}+\operatorname{Lip}_x(c),\\
\operatorname{Lip}(\mathcal S\psi)
&\leq\|h_\psi\|_{C(\mathcal X)}+\operatorname{Lip}_y(c).
\end{aligned}
\label{eq:update-lipschitz-bounds}
\end{equation}
\end{proposition}

\begin{proof}
Fix \(\psi\in C(\mathcal Y)\) with \(\|\psi\|_{C(\mathcal Y)}\leq M\),
and put \(D:=\operatorname{diam}(\mathcal Y)\) and
\(C_c:=\|c\|_{C(\mathcal X\times\mathcal Y)}\). Use the functions
\(F_\psi\) and \(Q_{x,p}\) from the proof of
\cref{prop:compact-well-posedness}. That proposition gives
\[
\|h_\psi\|_{C(\mathcal X)}\leq C_h(1+M)=:B_M.
\]
For \(x\in\mathcal X\) and \(|p|\leq B_M\),
\[
\frac{\mathrm dQ_{x,p}}{\mathrm d\nu}
\geq e^{-2M-2B_MD-2C_c}=:a_M>0.
\]
Consequently, for every unit vector \(e\),
\begin{align*}
e^\top\operatorname{Cov}_{Q_{x,p}}(Y)e
&=\inf_{t\in\mathbb R}\int_\mathcal Y(e\cdot y-t)^2\,Q_{x,p}(\mathrm dy)\\
&\geq a_M\inf_{t\in\mathbb R}\int_\mathcal Y(e\cdot y-t)^2\,\nu(\mathrm dy)
\geq a_M\sigma_\nu=:\lambda_M>0.
\end{align*}
Thus \(\nabla_{pp}^2F_\psi(x,p)\succeq\lambda_M I_d\) for \(x\in\mathcal X\) and \(|p|\leq B_M\).

The factor \(e^{p\cdot x}\) cancels in the normalization of
\(Q_{x,p}\). Comparing the remaining cost terms gives
\[
e^{-2\operatorname{Lip}_x(c)|x-x'|}
\leq\frac{\mathrm dQ_{x,p}}{\mathrm dQ_{x',p}}
\leq e^{2\operatorname{Lip}_x(c)|x-x'|}.
\]
Subtracting a fixed \(y_0\in\mathcal Y\) when comparing the means,
we obtain
\[
\left|\int y\,Q_{x,p}(\mathrm dy)-\int y\,Q_{x',p}(\mathrm dy)\right|
\leq D\bigl(e^{2\operatorname{Lip}_x(c)|x-x'|}-1\bigr).
\]
Using \(e^a-1\leq ae^a\) for \(a\geq0\), it follows that
\[
|\nabla_pF_\psi(x,p)-\nabla_pF_\psi(x',p)|
\leq K_x|x-x'|,
\quad
K_x:=1+2D\operatorname{Lip}_x(c)
 e^{2\operatorname{Lip}_x(c)\operatorname{diam}(\mathcal X)}.
\]
Set \(p=h_\psi(x)\) and \(p'=h_\psi(x')\). Both belong to
\(\overline B(0,B_M)\). The covariance bound and the two first-order
conditions give
\begin{align*}
\lambda_M|p-p'|^2
&\leq\bigl(\nabla_pF_\psi(x,p)-\nabla_pF_\psi(x,p')\bigr)\cdot(p-p')\\
&=\bigl(\nabla_pF_\psi(x',p')-\nabla_pF_\psi(x,p')\bigr)\cdot(p-p')\\
&\leq K_x|x-x'||p-p'|.
\end{align*}
Hence \(\operatorname{Lip}(h_\psi)\leq K_x/\lambda_M\).

The $\varphi$ and $\psi$-update \cref{eq:varphi-update,eq:psi-update} give
\[
\|\varphi_\psi\|_{C(\mathcal X)}\leq M+DB_M+C_c,
\qquad
\|\mathcal S\psi\|_{C(\mathcal Y)}
\leq\|\varphi_\psi\|_{C(\mathcal X)}+DB_M+C_c.
\]
To estimate \(\varphi_\psi\), use the minimizing property
\(\varphi_\psi(x)=\min_pF_\psi(x,p)\). For fixed \(p\),
\[
|F_\psi(x,p)-F_\psi(x',p)|
\leq\bigl(|p|+\operatorname{Lip}_x(c)\bigr)|x-x'|.
\]
Testing the minimum at \(x\) with \(p=h_\psi(x')\) gives
\[
\varphi_\psi(x)-\varphi_\psi(x')
\leq\bigl(\|h_\psi\|_{C(\mathcal X)}+\operatorname{Lip}_x(c)\bigr)
|x-x'|.
\]
Interchanging \(x\) and \(x'\) proves the first bound in
\cref{eq:update-lipschitz-bounds}. Comparing $\mathcal{S} \psi$ at
\(y,y'\) gives the second bound. Combining these estimates proves
the proposition.
\end{proof}

We will prove the existence of an optimiser for the EMOT problem by using a compactness argument and the exponential representation of \(\pi^{n+2,n+1}\). In the fixed gauge
\(\ell_n\), we use the first Sinkhorn update to obtain the regularity of
\(\psi^{n+1,[n]}\), and the second update to get the regularity of
\(\varphi^{n+2,[n]}\) and \(h^{n+2,[n]}\). For that reason in the following corollary we
keep all these functions in the same gauge.

\begin{corollary}[Uniform two-step estimates in a fixed gauge]
\label{prop:uniform-two-step}
Under \cref{ass:main,ass:two-step}, there is a constant \(C<\infty\),
independent of \(n\), such that, for every \(n\geq0\), for a good gauge
\(\ell_n\) of \cref{def:good-gauges} we have
\[
\sum_{k=n+1}^{n+2}
\left[
\begin{aligned}
&\lVert h^{k,[n]}\rVert_{C(\mathcal{X})}
+
\operatorname{Lip}\bigl(h^{k,[n]}\bigr)
+
\lVert\varphi^{k,[n]}\rVert_{C(\mathcal{X})}
\\
&\qquad
+
\operatorname{Lip}\bigl(\varphi^{k,[n]}\bigr)
+
\lVert\psi^{k,[n]}\rVert_{C(\mathcal{Y})}
+
\operatorname{Lip}\bigl(\psi^{k,[n]}\bigr)
\end{aligned}
\right]
\leq C.
\]
\end{corollary}

\begin{proof}
Fix \(n\geq0\) and keep \(\ell_n\) fixed. By
\cref{def:good-gauges}, \(\|\psi^{n,[n]}\|_{C(\mathcal Y)}\leq C_\psi\).
Using \cref{prop:gauge-preserves-sinkhorn}, apply
\cref{prop:multiplier-lipschitz} to this input. The sum of the six
norms at level \(n+1\) is at most \(C_1:=C_{C_\psi}\).
In particular, \(\|\psi^{n+1,[n]}\|_{C(\mathcal Y)}\leq C_1\).
A second application, now to \(\psi^{n+1,[n]}\), bounds the sum at
level \(n+2\) by \(C_2:=C_{C_1}\). The result follows with
\(C=C_1+C_2\), independently of \(n\).
\end{proof}

We can now extract convergent potentials. The marginal convergence result
identifies the terminal marginal of the limit, and the optimality
criterion from \cref{thm:general-pythagorean} finishes the proof.

\begin{proof}[Proof of \cref{thm:optimizer_representation}]
Fix good gauges \(\ell_n\) from \cref{def:good-gauges}. By
\cref{def:affine-gauges}, the intermediate densities can be written as
\[
\frac{\mathrm{d}\pi^{n+2,n+1}}{\mathrm{d}(\mu\otimes\nu)}(x,y)
=
\exp\!\left(
-\varphi^{n+2,[n]}(x)
-\psi^{n+1,[n]}(y)
-h^{n+2,[n]}(x)\cdot(y-x)
-c(x,y)
\right).
\]

By \cref{prop:uniform-two-step}, the families
\[
\bigl(\varphi^{n+2,[n]}\bigr)_{n\geq 0},
\qquad
\bigl(\psi^{n+1,[n]}\bigr)_{n\geq 0},
\qquad
\bigl(h^{n+2,[n]}\bigr)_{n\geq 0}
\]
are uniformly bounded and equi-Lipschitz on \(\mathcal{X}\), \(\mathcal{Y}\), and \(\mathcal{X}\),
respectively. Hence, by the Arzel\`a--Ascoli theorem, there exist a
subsequence \((n_j)_{j\geq 1}\) and functions
\[
\varphi^\star\in\operatorname{Lip}(\mathcal{X}),\qquad
\psi^\star\in\operatorname{Lip}(\mathcal{Y}),\qquad
h^\star\in\operatorname{Lip}(\mathcal{X};\mathbb{R}^d),
\]
such that
\[
\varphi^{n_j+2,[n_j]}\longrightarrow\varphi^\star
\quad\text{uniformly on \(\mathcal{X}\)},
\]
\[
\psi^{n_j+1,[n_j]}\longrightarrow\psi^\star
\quad\text{uniformly on \(\mathcal{Y}\)},
\]
and
\[
h^{n_j+2,[n_j]}\longrightarrow h^\star
\quad\text{uniformly on \(\mathcal{X}\)}.
\]

Define the measure \(\pi^\star\) by
\[
\frac{\mathrm d\pi^\star}{\mathrm d(\mu\otimes\nu)}(x,y)
:=
\exp\!\left(
-\varphi^\star(x)
-\psi^\star(y)
-h^\star(x)\cdot(y-x)
-c(x,y)
\right).
\]
The densities of \(\pi^{n_j+2,n_j+1}\) relative to \(\mu\otimes\nu\)
converge uniformly to this density on \(\mathcal X\times\mathcal Y\).
Its integral is therefore one, so \(\pi^\star\) is a probability
measure. Moreover,
\[
\pi^{n_j+2,n_j+1}\longrightarrow\pi^\star
\]
in total variation.

We verify that \(\pi^\star\in\mathcal{M}(\mu,\nu)\). Since
\(\pi^{n_j+2,n_j+1}\in\mathcal{M}(\mu,\cdot)\), its first marginal is
\(\mu\). Passing to the limit in total variation shows that the first
marginal of \(\pi^\star\) is also \(\mu\).

Moreover, the second marginal of \(\pi^{n_j+2,n_j+1}\) is
\(\nu^{n_j+2,n_j+1}\). By \cref{thm:main-convergence},
\[
H\!\left(\nu^{n_j+2,n_j+1}\mid\nu\right)\longrightarrow 0.
\]
Hence, by Pinsker's inequality,
\[
\nu^{n_j+2,n_j+1}\longrightarrow\nu
\]
in total variation. Since total variation convergence is preserved by
taking marginals, the second marginal of \(\pi^\star\) is \(\nu\).

Finally, for every bounded measurable
\(\zeta:\mathcal{X}\to\mathbb{R}^d\),
\[
\int_{\mathcal{X}\times \mathcal{Y}}
\zeta(x)\cdot(y-x)\,
\pi^{n_j+2,n_j+1}(dx,dy)
=0.
\]
Since \(\mathcal{X}\) and \(\mathcal{Y}\) are compact, the integrand is bounded. Passing to
the limit gives
\[
\int_{\mathcal{X}\times \mathcal{Y}}
\zeta(x)\cdot(y-x)\,
\pi^\star(dx,dy)
=0.
\]
Thus \(\pi^\star\in\mathcal{M}(\mu,\nu)\).

By the definition of \(R\),
\[
\frac{\mathrm{d}\pi^\star}{\mathrm{d}R}(x,y)
=
\exp\!\left(
-\varphi^\star(x)
-\psi^\star(y)
-h^\star(x)\cdot(y-x)
\right).
\]
The potentials are bounded on their compact domains.
Theorem~\ref{thm:general-pythagorean} therefore gives
\(H(\pi^\star\mid R)<\infty\), optimality, and uniqueness.
Mapping back through the affine isometry in
\cref{rem:compact-support-consequences} completes the proof in the
original coordinates.
\end{proof}

\section{A conditional stability estimate}
\label{sec:stability}

In this section we compare two martingale couplings with the same first marginal and
different terminal marginals. The estimate is conditional on density
and potential bounds, rather than on assumptions on \(\mu,\nu,\nu'\)
alone. Only the hypotheses below are used; no convexity of the terminal
support or regularity of a reference cost is required. In
\cref{sec:exponential-convergence}, the potential estimates from
\cref{sec:two-step-estimates} are used to verify these hypotheses uniformly along
the Sinkhorn iteration, providing the bound needed for exponential
convergence.

\begin{assumption}[Density and potential bounds]
\label{ass:stability}
Let \(\mathcal X,\mathcal Y\subset\mathbb R^d\) be compact,
\(\mu\in\mathcal P(\mathcal X)\), and
\(\nu',\nu\in\mathcal P(\mathcal Y)\) satisfy $\mu\preceq_{\mathrm{cx}}\nu'$ and $\mu\preceq_{\mathrm{cx}}\nu$.
Let \(\pi'\in\mathcal M(\mu,\nu')\) and
\(\pi\in\mathcal M(\mu,\nu)\) satisfy the following conditions.
\begin{enumerate}
\item The coupling \(\pi'\) admits a density satisfying
\[
\frac{\mathrm d\pi'}{\mathrm d(\mu\otimes\nu')}
\geq q>0
\qquad\mu\otimes\nu'\text{-almost everywhere},
\]
where \(q\) is a constant. Necessarily \(q\leq1\), since the density
has integral one.
\item There are finite Borel functions \(a,g,b\), with
\[
a\in L^\infty(\mathcal X,\mu),\qquad
g\in L^\infty(\mathcal Y,\nu'),\qquad
b\in L^\infty(\mathcal X,\mu;\mathbb R^d),
\]
such that
\begin{equation}
\begin{aligned}
\frac{\mathrm d\pi}{\mathrm d\pi'}(x,y)&=e^{w(x,y)},\\
w(x,y)&=a(x)+g(y)+b(x)\cdot(y-x),\qquad
\|w\|_{L^\infty(\mathcal X\times\mathcal Y,\pi')}\leq B,
\end{aligned}
\label{eq:bounded-log-comparison}
\end{equation}
where \(0\leq B<\infty\), and the density identity holds
\(\pi'\)-almost everywhere.
\end{enumerate}
\end{assumption}

The exponential comparison implies \(\pi\sim\pi'\), and hence
\(\nu\sim\nu'\).
If the two couplings have bounded potential representations relative to
a common reference probability measure \(R\), the choices
\(a=\varphi'-\varphi\), \(g=\psi'-\psi\), and \(b=h'-h\)
give \cref{eq:bounded-log-comparison}. Their optimality for their
respective terminal marginals then follows from
\cref{thm:general-pythagorean}.

By \cref{ass:stability} both terminal marginals $\nu$ and $\nu'$ have the same mass and barycenter, so their
difference annihilates affine functions. For this reason, the proof of the result of this section
uses an estimate of the potential \(g\) up to an affine function.

\begin{theorem}[Stability in relative entropy]
\label{thm:stability}
Under \cref{ass:stability},
\begin{equation}
H(\pi\mid\pi')+H(\pi'\mid\pi)
\leq\frac{2e^{2B}}{q}H(\nu\mid\nu').
\label{eq:symmetrized-stability-estimate}
\end{equation}
In particular, the same bound holds for \(H(\pi\mid\pi')\).
\end{theorem}

\begin{proof}
\emph{Entropy cancellation.}
Set \(s=\mathrm d\nu/\mathrm d\nu'-1\). For a random pair \((X,Y)\)
with law \(\pi'\), conditioning the density identity  of \cref{eq:bounded-log-comparison} on \(Y\) gives
\begin{equation}
1+s(Y)=\mathbb E_{\pi'}[e^{w(X,Y)}\mid Y],\qquad \text{so that} \quad
e^{-B}\leq1+s\leq e^B\quad\nu'\text{-almost everywhere}.
\label{eq:stability-marginal-ratio}
\end{equation}
Thus \(s\in L^\infty(\mathcal Y,\nu')\), and all entropies in the further calculations are
finite. The two terminal marginals have the same mass and barycenter,
because both couplings are martingales with first marginal \(\mu\).
Consequently,
\[
\int_\mathcal Y s\ell\,\mathrm d\nu'=0
\qquad\text{for every affine function }\ell.
\]
Write \(\mathcal A(\mathcal Y)\) for the restrictions to
\(\mathcal Y\) of affine functions on \(\mathbb R^d\), and set
\[
J:=H(\pi\mid\pi')+H(\pi'\mid\pi),\qquad
d_g:=\inf_{\ell\in\mathcal{A}(\mathcal Y)}
\|g-\ell\|_{L^2(\nu')}.
\]
The boundedness of \(a,g,b\) and the compact supports justify integrating
the three terms in \(w\) separately. The initial-marginal and martingale
constraints for $\pi$ and $\pi'$ cancel the \(a\) and \(b\) terms, so
\begin{align}
J&=\int w\,\mathrm d(\pi-\pi')
 =\int_\mathcal Y g\,\mathrm d(\nu-\nu')\notag\\
 &=\int_\mathcal Y(g-\ell)s\,\mathrm d\nu'
 \leq\|g-\ell\|_{L^2(\nu')}\|s\|_{L^2(\nu')}
 \qquad \text{for} \quad \ell\in\mathcal{A}(\mathcal Y).
\label{eq:entropy-cancellation}
\end{align}
Taking the infimum gives \(J\leq d_g\|s\|_{L^2(\nu')}\).

\emph{Estimate modulo affine functions.}
For every real \(z\) with \(|z|\leq B\) we have
\[
(e^z-1)z=z^2\int_0^1e^{tz}\,\mathrm dt\geq e^{-B}z^2.
\]
Since \(\pi'\sim\mu\otimes\nu'\), the representation of \(w\)
also holds \(\mu\otimes\nu'\)-almost everywhere. Fubini's theorem
and Assumption~\ref{ass:stability}(1) therefore yield
\begin{align}
J&=\int(e^w-1)w\,\mathrm d\pi'
 \geq e^{-B}\int w^2\,\mathrm d\pi'\notag\\
 &\geq e^{-B}q\int_\mathcal X\int_\mathcal Y
 |g(y)+a(x)+b(x)\cdot(y-x)|^2
 \,\nu'(\mathrm dy)\mu(\mathrm dx)\notag\\
 &\geq e^{-B}q\,d_g^2.
\label{eq:entropy-affine-coercivity}
\end{align}
The last step uses that, for \(\mu\)-almost every fixed \(x\), the
function \(y\mapsto-a(x)-b(x)\cdot(y-x)\) is affine.
If \(d_g=0\), the preceding upper bound on \(J\) gives \(J=0\).
Otherwise, combining that bound with
\cref{eq:entropy-affine-coercivity} gives
\(d_g\leq e^Bq^{-1}\|s\|_{L^2(\nu')}\). In both cases,
\begin{equation}
J\leq\frac{e^B}{q}\|s\|_{L^2(\nu')}^2.
\label{eq:stability-chi-square-bound}
\end{equation}

\emph{Comparison with the marginal relative entropy.}
For \(r>-1\), put \(\chi(r)=(1+r)\log(1+r)-r\).
Since \(\chi(0)=\chi'(0)=0\) and \(\chi''(r)=(1+r)^{-1}\),
Taylor's formula gives
\[
\chi(r)=r^2\int_0^1\frac{1-t}{1+tr}\,\mathrm dt
\geq\frac{e^{-B}}2r^2
\qquad\bigl(-1<r\leq e^B-1\bigr).
\]
Indeed, \(0<1+tr\leq e^B\) along the integration segment.
Using \cref{eq:stability-marginal-ratio} and
\(\int s\,\mathrm d\nu'=0\), we obtain
\[
H(\nu\mid\nu')=\int_\mathcal Y\chi(s)\,\mathrm d\nu'
\geq\frac{e^{-B}}2\|s\|_{L^2(\nu')}^2.
\]
Together with \cref{eq:stability-chi-square-bound}, this proves
\cref{eq:symmetrized-stability-estimate}. The one-sided bound follows
from \(H(\pi'\mid\pi)\geq0\).
\end{proof}

\section{Exponential convergence}
\label{sec:exponential-convergence}

The remaining step is to make the constants provided by \cref{thm:stability} for the Sinkhorn iterates
uniform along the iteration. We take \(\pi'=\pi^{n+1,n}\) and
\(\pi=\pi^\star\), so that the marginal entropy in the stability
estimate is the term identified in \cref{eq:exact-marginal-decrement}.
Once the density and log-density ratio bounds are verified, the entropy
decreases by a fixed fraction over every full iteration.

For the density lower bound, the product measure to consider is
\(\mu\otimes\nu^{n+1,n}\), not \(\mu\otimes\nu\). Passing from one to
the other replaces \(\psi^{n,[n]}\) by \(\psi^{n+1,[n]}\) in the
exponential representation of the density.

\begin{proposition}[Verification of the stability assumptions]
\label{prop:verify-stability-conditions}
Under Assumption~\ref{ass:main} and \cref{ass:two-step}, for every
\(n\geq0\) the choices
\[
\pi'=\pi^{n+1,n},\qquad
\nu'=\nu^{n+1,n},\qquad
\pi=\pi^\star
\]
satisfy \cref{ass:stability} with constants \(q>0\) and \(0\leq B<\infty\)
independent of \(n\).
\end{proposition}

\begin{proof}
Fix \(n\geq0\) and use the good gauge \(\ell_n\) from
\cref{def:good-gauges}.
By this choice, \(\psi^{n,[n]}\) is uniformly bounded, while
\(\varphi^{n+1,[n]},h^{n+1,[n]},\psi^{n+1,[n]}\) are uniformly
bounded by \cref{prop:uniform-two-step}.
By \cref{eq:intermediate-marginal-density},
\(\nu^{n+1,n}\sim\nu\), and
\[
\begin{aligned}
\frac{\mathrm d\pi^{n+1,n}}
{\mathrm d(\mu\otimes\nu^{n+1,n})}(x,y)
=\exp\Bigl(&-\varphi^{n+1,[n]}(x)-\psi^{n+1,[n]}(y)\\
&-h^{n+1,[n]}(x)\cdot(y-x)-c(x,y)\Bigr).
\end{aligned}
\]
Since \(c\) is bounded and the supports are compact, this density has
a lower bound \(q>0\) independent of \(n\).

Fix the optimizer potentials from \cref{thm:optimizer_representation}
and, for this \(n\), take
\[
a=\varphi^{n+1,[n]}-\varphi^\star,\qquad
g=\psi^{n,[n]}-\psi^\star,\qquad
b=h^{n+1,[n]}-h^\star.
\]
Both couplings are equivalent to \(R\), and their log-density ratio is
\[
\log\frac{\mathrm d\pi^\star}{\mathrm d\pi^{n+1,n}}(x,y)
=a(x)+g(y)+b(x)\cdot(y-x).
\]
The functions \(a,g,b\) are uniformly bounded. Compactness gives
\[
\left\|\log\frac{\mathrm d\pi^\star}{\mathrm d\pi^{n+1,n}}\right\|_
{L^\infty(\mathcal X\times\mathcal Y,\pi^{n+1,n})}\leq B<\infty
\]
with \(B\) independent of \(n\). This verifies \cref{ass:stability}.
\end{proof}

We now combine the uniform stability estimate with the full-step
entropy decrement.

\begin{proof}[Proof of \cref{thm:exponential-halfstep}]
By \cref{thm:optimizer_representation},
\(H(\pi^\star\mid R)<\infty\), so
\cref{prop:half-step-identities} applies with \(\bar\pi=\pi^\star\).
Combining its full-step decrement and exact marginal identity with
\cref{thm:stability,prop:verify-stability-conditions} yields
\[
\begin{aligned}
&H(\pi^\star\mid\pi^{n+1,n})
 -H(\pi^\star\mid\pi^{n+2,n+1})\\
&\qquad\geq H(\nu\mid\nu^{n+1,n})
 \geq\frac{q}{2e^{2B}}H(\pi^\star\mid\pi^{n+1,n}).
\end{aligned}
\]
Since \(0<q\leq1\) and \(B\geq0\), the number
\[
\alpha:=1-\frac{q}{2e^{2B}}
\]
belongs to \((0,1)\), and
\[
H(\pi^\star\mid\pi^{n+2,n+1})
\leq\alpha H(\pi^\star\mid\pi^{n+1,n}).
\]
Iteration proves \cref{eq:exponential-halfstep}.
\end{proof}

\appendix

\section{Auxiliary proofs}
\label{app:auxiliary-proofs}

The following lemma gives the uniform ball-mass and covariance bounds
used in \cref{sec:two-step-estimates}.

\begin{lemma}[Nondegeneracy from compact full support]
\label{lem:compact-support-nondegeneracy}
Let \(\eta\in\mathcal P(\mathbb R^d)\) and let
\(S:=\operatorname{supp}\eta\) be compact. Then, for every fixed \(r>0\),
\[
m_\eta(r):=\inf_{z\in S}\eta(B(z,r))>0.
\]
If, in addition, \(\operatorname{aff}(S)=\mathbb R^d\), then
\[
\sigma_\eta:=
\lambda_{\min}\!\left(
\int_S (z-\bar z_\eta)(z-\bar z_\eta)^\top\,\eta(\mathrm dz)
\right)>0,
\qquad \text{where }
\bar z_\eta:=\int_S z\,\eta(\mathrm dz),
\]
and $\lambda_{\min}$ denotes the minimal eigenvalue of the matrix.
\end{lemma}

\begin{proof}
Fix \(r>0\). By compactness, finitely many balls
\(B(z_i,r/2)\), with \(z_i\in S\), cover \(S\).
Since \(z_i\in\operatorname{supp}\eta\),
\(\eta(B(z_i,r/2))>0\) for every \(i\). For any \(z\in S\), choose
\(i\) with \(|z-z_i|<r/2\). Then \(B(z_i,r/2)\subset B(z,r)\), and hence
\[
\eta(B(z,r))
\geq
\min_i\eta(B(z_i,r/2))
>0.
\]
Taking the infimum over \(z\in S\) proves \(m_\eta(r)>0\).

Assume now that \(\operatorname{aff}(S)=\mathbb R^d\).
The covariance matrix is finite because \(S\) is compact. If it were not
positive definite, there would exist \(v\neq0\) such that
\[
\int_S \bigl(v\cdot(z-\bar z_\eta)\bigr)^2\,\eta(\mathrm dz)=0.
\]
Thus \(v\cdot z=v\cdot\bar z_\eta\) for \(\eta\)-almost every \(z\).
The function \(z\mapsto v\cdot z-v\cdot\bar z_\eta\) is continuous, and
\(S=\operatorname{supp}\eta\); therefore it vanishes on all of \(S\).
Hence \(S\) is contained in an affine hyperplane, contradicting
\(\operatorname{aff}(S)=\mathbb R^d\). Thus the covariance matrix is
positive definite and \(\sigma_\eta>0\).
\end{proof}

% Embedded bibliography: compile this file directly with pdfLaTeX.


\begin{thebibliography}{24}

\bibitem{AlfonsiCoyaudEhrlacherLombardi2021MomentConstraintsOT}
Aur\'elien Alfonsi, Rafa\"el Coyaud, Virginie Ehrlacher, and Damiano Lombardi.
\newblock Approximation of optimal transport problems with marginal moments constraints.
\newblock \emph{Mathematics of Computation}, 90(328):689--737, 2021.

\bibitem{BackhoffVeraguasPammer2022StabilityMOTWeakOT}
Julio Backhoff-Veraguas and Gudmund Pammer.
\newblock Stability of martingale optimal transport and weak optimal transport.
\newblock \emph{The Annals of Applied Probability}, 32(1):721--752, 2022.

\bibitem{BeiglbockHenryLaborderePenkner2013ModelIndependentBounds}
Mathias Beiglb\"ock, Pierre Henry-Labord\`ere, and Friedrich Penkner.
\newblock Model-independent bounds for option prices---a mass transport approach.
\newblock \emph{Finance and Stochastics}, 17(3):477--501, 2013.

\bibitem{BenamouChazareixHoffmannLoeperVialard2024EntropicSemiMartingaleOT}
Jean-David Benamou, Guillaume Chazareix, Marc Hoffmann, Gr\'egoire Loeper, and Fran\c{c}ois-Xavier Vialard.
\newblock Entropic semi-martingale optimal transport.
\newblock arXiv preprint arXiv:2408.09361, 2024.

\bibitem{CarlierMalamutSylvestre2025WeakOTMomentConstraints}
Guillaume Carlier, Hugo Malamut, and Maxime Sylvestre.
\newblock Weak optimal transport with moment constraints: constraint qualification, dual attainment and entropic regularization.
\newblock arXiv preprint arXiv:2511.16211, 2025.

\bibitem{ChenConfortiRenWang2026SinkhornEMOT}
Fan Chen, Giovanni Conforti, Zhenjie Ren, and Xiaozhen Wang.
\newblock Convergence of Sinkhorn's algorithm for entropic martingale optimal transport problem.
\newblock \emph{Mathematics of Operations Research}, 2026.
\newblock doi: \href{https://doi.org/10.1287/moor.2024.0619}{10.1287/moor.2024.0619}.

\bibitem{chiarini2024semiconcavity}
Alberto Chiarini, Giovanni Conforti, Giacomo Greco, and Luca Tamanini.
\newblock A semiconcavity approach to stability of entropic plans and exponential convergence of Sinkhorn's algorithm.
\newblock arXiv preprint arXiv:2412.09235, 2024. Accepted in The Annals of Probability.

\bibitem{DeMarch2018EntropicApproximationMOT}
Hadrien De March.
\newblock Entropic approximation for multi-dimensional martingale optimal transport.
\newblock arXiv preprint arXiv:1812.11104, 2018.

\bibitem{DeMarchHenryLabordere2019ArbitrageFreeImpliedVolatility}
Hadrien De March and Pierre Henry-Labord\`ere.
\newblock Building arbitrage-free implied volatility: Sinkhorn's algorithm and variants.
\newblock arXiv preprint arXiv:1902.04456, 2019.

\bibitem{DoldiFrittelli2023EMOT}
Alessandro Doldi and Marco Frittelli.
\newblock Entropy martingale optimal transport and nonlinear pricing--hedging duality.
\newblock \emph{Finance and Stochastics}, 27(2):255--304, 2023.

\bibitem{DoldiFrittelliRosazzaGianin2024EntropyMOTTheory}
Alessandro Doldi, Marco Frittelli, and Emanuela Rosazza Gianin.
\newblock On entropy martingale optimal transport theory.
\newblock \emph{Decisions in Economics and Finance}, 47(1):1--42, 2024.

\bibitem{EcksteinGuoLimObloj2021RobustPricingMultipleAssets}
Stephan Eckstein, Gaoyue Guo, Tongseok Lim, and Jan Ob\l\'oj.
\newblock Robust pricing and hedging of options on multiple assets and its numerics.
\newblock \emph{SIAM Journal on Financial Mathematics}, 12(1):158--188, 2021.

\bibitem{eckstein2025exponential}
Stephan Eckstein and Aziz Lakhal.
\newblock Exponential convergence of general iterative proportional fitting procedures.
\newblock \emph{SIAM Journal on Optimization}, 36(2):912--937, 2026.
\newblock doi: \href{https://doi.org/10.1137/25M1752584}{10.1137/25M1752584}.

\bibitem{GalichonHenryLabordereTouzi2014StochasticControl}
Alfred Galichon, Pierre Henry-Labord\`ere, and Nizar Touzi.
\newblock A stochastic control approach to no-arbitrage bounds given marginals, with an application to lookback options.
\newblock \emph{The Annals of Applied Probability}, 24(1):312--336, 2014.

\bibitem{GuoObloj2019ComputationalMOT}
Gaoyue Guo and Jan Ob\l\'oj.
\newblock Computational methods for martingale optimal transport problems.
\newblock \emph{The Annals of Applied Probability}, 29(6):3311--3347, 2019.

\bibitem{Guyon2024DispersionConstrainedMartingaleSchrodinger}
Julien Guyon.
\newblock Dispersion-constrained martingale Schr\"odinger problems and the exact joint S\&P 500/VIX smile calibration puzzle.
\newblock \emph{Finance and Stochastics}, 28(1):27--79, 2024.

\bibitem{HenryLabordere2019MartingaleSchrodingerBridges}
Pierre Henry-Labord\`ere.
\newblock From (martingale) Schr\"odinger bridges to a new class of stochastic volatility models.
\newblock arXiv preprint arXiv:1904.04554, 2019.

\bibitem{HiewNennaPass2025ODEEntropicOT}
Joshua Zoen-Git Hiew, Luca Nenna, and Brendan Pass.
\newblock An ordinary differential equation for entropic optimal transport and its linearly constrained variants.
\newblock \emph{Numerische Mathematik}, 157(6):2211--2249, 2025.

\bibitem{NutzWiesel2024MartingaleSchrodingerBridgeTwoDistributions}
Marcel Nutz and Johannes Wiesel.
\newblock On the martingale Schr\"odinger bridge between two distributions.
\newblock arXiv preprint arXiv:2401.05209, 2024. Accepted in Bernoulli.

\bibitem{NutzWieselZhao2023MartingaleSchrodingerBridges}
Marcel Nutz, Johannes Wiesel, and Long Zhao.
\newblock Martingale Schr\"odinger bridges and optimal semistatic portfolios.
\newblock \emph{Finance and Stochastics}, 27(1):233--254, 2023.

\bibitem{Strassen1965}
Volker Strassen.
\newblock The existence of probability measures with given marginals.
\newblock \emph{The Annals of Mathematical Statistics}, 36(2):423--439, 1965.

\bibitem{TanTouzi2013OptimalTransportationControlledDynamics}
Xiaolu Tan and Nizar Touzi.
\newblock Optimal transportation under controlled stochastic dynamics.
\newblock \emph{The Annals of Probability}, 41(5):3201--3240, 2013.

\bibitem{TangShavlovskyRahmanianXiaoYing2025EntropicOTMartingaleType}
Xun Tang, Michael Shavlovsky, Holakou Rahmanian, Tesi Xiao, and Lexing Ying.
\newblock An efficient algorithm for entropic optimal transport under martingale-type constraints.
\newblock arXiv preprint arXiv:2508.17641, 2025.

\bibitem{Wiesel2023ContinuityMOTRealLine}
Johannes Wiesel.
\newblock Continuity of the martingale optimal transport problem on the real line.
\newblock \emph{The Annals of Applied Probability}, 33(6A):4645--4692, 2023.



\end{thebibliography}
\end{document}